\documentclass[a4paper,12pt]{amsart}
\usepackage{hyperref}
\usepackage{amssymb}
\usepackage{color}

\allowdisplaybreaks

\numberwithin{equation}{section}

\newtheorem{thm}{Theorem}[section]
\newtheorem{prop}[thm]{Proposition}
\newtheorem{lem}[thm]{Lemma}
\newtheorem{cor}[thm]{Corollary}

\theoremstyle{definition}

\theoremstyle{remark}
\newtheorem{rmk}[thm]{Remark}
\newtheorem{ex}[thm]{Example}

\newcommand{\C}{\mathbb{C}}
\newcommand{\N}{\mathbb{N}}
\newcommand{\R}{\mathbb{R}}
\newcommand{\T}{\mathbb{T}}
\newcommand{\Z}{\mathbb{Z}}

\def\LL{\mathcal{L}}
\def\KK{\mathcal{K}}

\def\HH{\mathcal{H}}

\newcommand{\Inf}{\operatorname{Inf}}

\newcommand{\id}{\operatorname{id}}

\newcommand{\Aut}{\operatorname{Aut}}

\newcommand{\lsp}{\operatorname{span}}
\newcommand{\clsp}{\overline{\lsp}}

\newcommand{\KP}{{\operatorname{KP}}}

\title[Deformations of Fell bundles]{\boldmath{Deformations of Fell bundles and\\
twisted graph algebras}}
\author[Iain Raeburn]{Iain Raeburn}

\address{Iain Raeburn\\ Department of Mathematics and Statistics\\University of Otago\\PO Box 56\\Dunedin 9054\\New Zealand}
\email{iraeburn@maths.otago.ac.nz}
\date{16 November 2013; revised 14 March 2016}

\thanks{This research has been supported by the Marsden Fund of the Royal Society of New Zealand.}

\begin{document}

\begin{abstract}
We consider Fell bundles over discrete groups, and the $C^*$-algebra which is universal for representations of the bundle. We define deformations of Fell bundles, which are new Fell bundles with the same underlying Banach bundle but with the multiplication deformed by a two-cocycle on the group. Every graph algebra can be viewed as the $C^*$-algebra of a Fell bundle, and there are are many cocycles of interest with which to deform them. We thus obtain many of the twisted graph algebras of Kumjian, Pask and Sims. We demonstate the utility of our approach to these twisted graph algebras by proving that the deformations associated to different cocycles can be assembled as the fibres of a $C^*$-bundle.
\end{abstract}

\maketitle

\section{Introduction}

Higher-rank graphs (or \emph{$k$-graphs}) are higher-dimensional analogues of directed graphs. They were invented by Kumjian and Pask \cite{KP} as combinatorial models for the higher-rank Cuntz--Krieger algebras of \cite{RSteg}. There has since been considerable interest in the $C^*$-algebras of $k$-graphs, and the class of $k$-graph algebras contains many interesting $C^*$-algebras (see \cite{PRRS, DY, BSV}, for example).

Directed graphs, which are essentially the same as $1$-graphs, can be visualised as a $1$-dimensional simplicial complex, and then the theory of covering spaces has a useful and elegant combinatorial version (see \cite{DPR}, for example). Higher-rank graphs can also be realised topologically, and the theory of covering spaces carries over in a very satisfactory way \cite{PQR, KKQS}. More recently, Kumjian, Pask and Sims have considered homology and cohomology for $k$-graphs, and have studied a family of twisted graph algebras in which the usual relations are deformed by a $2$-cocycle \cite{KPS1, KPS2, KPS3}. 

Here we show that graph algebras can be viewed as the universal $C^*$-algebras of Fell bundles over discrete groups. We then deform the multiplication on these bundles using the $2$-cocycles arising in group cohomology, and view the $C^*$-algebras of the deformed bundles as twisted graph algebras. This construction is not quite as general as those in \cite{KPS1, KPS2}, but it is general enough to cover all the main examples, and it seems more accessible.

We think of Fell bundles over a group $G$ as $C^*$-algebraic analogues of $G$-gradings (and we discuss the similarities and differences in Appendix~\ref{app:grade}). So it is no surprise that every $k$-graph algebra is given by a Fell bundle over $\Z^k$ in which the $0$-graded component is the core. But there are many other Fell bundles around: for example, every $k$-graph algebra is also given by a Fell bundle over the fundamental group of the $k$-graph. We can get this last Fell bundle by stringing together established results: the $k$-graph has a universal convering graph that can be realised as a skew product of the graph with its fundamental group \cite{KP2, PQR}; the $C^*$-algebra of the skew product can be viewed as the crossed product by a coaction of the fundamental group on the original graph algebra \cite{KQR, PQR}; and a theorem of Quigg \cite{Q} says that this coaction is associated to a Fell bundle structure for the graph algebra. But having done this, we find it is really quite easy to go directly to the Fell bundle, and this is what we do.

We begin by recalling the definitions of Fell bundles and their $C^*$-algebras. Then we review the group-theoretic cocycles of interest to us, and show how to deform Fell bundles using them. In \S\ref{sec:galgs} we apply this construction to bundles arising from graph algebras, and prove our first main theorem, which says that the $C^*$-algebras of the deformed bundles are generated by families of partial isometries satisfying a twisted version of the usual Cuntz--Krieger relations of \cite{KP}. In \S\ref{sec:unique} we prove versions of the uniqueness theorems for these twisted graph algebras. We also prove that the deformed Fell bundles are all amenable in the sense of Exel \cite{E}, and discuss criteria for simplicity. Then in \S\ref{sec:field} we prove that the $C^*$-algebras of the different deformations of a fixed Fell bundle over $G$ can be assembled as the fibres of a continuous $C^*$-bundle over the second cohomology group of $G$.

We close with three appendices which describe different aspects of the approach we have taken. In the first, we discuss the theorem of Quigg which says that Fell bundles over a discrete group are essentially the same thing as coactions of the group on $C^*$-algebras, and give a short and relatively elementary proof of his theorem. Next, we expand on the relationship between Fell bundles and gradings. We discuss how the relationship between an element of the $C^*$-algebra and its homogeneous components resembles that between a continuous function and its Fourier series, and use the analogy to prove a $C^*$-algebraic version of Fej\'er's theorem on the accelerated converence of Fourier series. In the third appendix, we use the ideas of the previous one to give a direct proof that Leavitt path algebras and Kumjian--Pask algebras are graded.

\subsection*{Acknowledgements}

My interest in these ideas arises from an invitation by Aidan Sims and Roozbeh Hazrat to speak at an AMSI workshop at the University of Western Sydney in February 2013. They asked me to explain, to a mixed audience of algebraists and analysts, how coactions of groups arose in the context of graph algebras.  I thank them for providing the stimulus for me to revisit this topic, though in retrospect I seem to have strayed off subject. (The parts of the subject relevant to their question are discussed in the appendices.) I also thank the participants in our seminar at Otago, and especially Lisa Clark, for their feedback and continued interest while I was straying.

\section{Fell bundles and their $C^*$-algebras}

Since the basic definitions are spread over the first 900 pages of \cite{FD}, and since the definitions simplify for the Fell bundles over discrete groups of interest to us here, we begin by giving a detailed definition.

A \emph{Banach bundle} over a (discrete) set $X$ is a set $B$ and a function $p:B\to X$ such that each $B_x:=p^{-1}(x)$ has the structure of a complex Banach space; we call $p$ the \emph{projection} for the bundle. We suppose that $p:B\to G$ is a Banach bundle over a group $G$, and that $B$ carries an associative multiplication and an involution satisfying
\begin{enumerate}
\item\label{FBa} $p(bc)=p(b)p(c)$ for all $b,c\in B$;
\item\label{FBb} for every $g,h\in G$ the map $(b,c)\mapsto bc$ is bilinear from $B_g\times B_h$ to $B_{gh}$, and satisfies $\|bc\|\leq \|b\|\,\|c\|$ for $b\in B_g$, $c\in B_h$;
\item\label{FBc} for every $g\in G$ the map $b\mapsto b^*$ is conjugate linear from $B_g$ to $B_{g^{-1}}$, and satisfies $(bc)^*=c^*b^*$ and $\|b^*b\|=\|b\|^2$.
\end{enumerate}
Then $B_e$ is a $C^*$-algebra. If in addition
\begin{enumerate}\setcounter{enumi}{3}
\item\label{FBd} $b^*b\geq 0$ in $B_e$ for all $b\in B$,
\end{enumerate}
then $p:B\to G$ is a \emph{Fell bundle} over $G$. (Fell bundles were called \emph{$C^*$-algebraic bundles} in \cite{FD} and \cite{Q}.) We write $\Gamma_c(B)$ for the vector-space direct sum $\bigoplus_{g\in G}B_g$. Provided we remember that adding and deleting zero elements $0_g$ does not change sums, we can view elements of $\Gamma_c(B)$ as finite formal sums $\sum_{g\in F} b_g$. Then extending the multiplication bilinearly and and the $*$-operation conjugate-linearly give a multiplication and involution on $\Gamma_c(B)$, which thus becomes a $*$-algebra.

A \emph{representation} $\pi$ of a Fell bundle $p:B\to G$ in a $C^*$-algebra $C$ consists of linear maps $\pi_g:B_g\to C$ such that $\pi_{gh}(bc)=\pi_g(b)\pi_h(c)$, $\pi_g(b)^*=\pi_{g^{-1}}(b^*)$, and $\pi_e$ is a nondegenerate representation of $B_e$. The $C^*$-algebra $C^*(B)$ is generated by a universal representation of $p:B\to G$, which we will denote by $\pi$. The maps $\pi_g$ combine to give a $*$-homomorphism $\sum_{g} b_g\mapsto \sum_{g}\pi_g(b_g)$ of $\Gamma_c(B)$ into $C^*(B)$, and its range is a dense $C^*$-subalgebra of $C^*(B)$.

We can view each $B_g$ as a right Hilbert module over $B_e$, and Exel proved that $B$ has a left regular representation  by adjointable operators on the Hilbert-module direct sum $\ell^2(B):=\bigoplus_{g\in G}B_g$ \cite[\S2]{E}. The \emph{reduced $C^*$-algebra} $C_r^*(B)$ is the $C^*$-subalgebra of $\LL(\ell^2(B))$ generated by the range of this representation. The action of $B_e$ on the summand $B_e$ in $\ell^2(B)$ is just the action by left multiplication. Thus the left regular representation is isometric on $B_e$, and in view of the relation $\|b^*b\|=\|b\|^2$ in \eqref{FBc}, it is isometric on each $B_g$. Thus $B$ embeds isometrically in both $C^*_r(B)$ and $C^*(B)$. When the left regular representation is faithful on $C^*(B)$, Exel says that $B$ is \emph{amenable}.

\section{Cocycles and deformations of Fell bundles}

Suppose that $G$ is a group and $Z$ is an abelian group. A \emph{$2$-cocycle} $\sigma$ on $G$ with coefficients in $Z$ is a function $\sigma:G\times G\to Z$ such that $\sigma(e,g)=\sigma(g,e)=1$ for all $g$ and 
\begin{equation}\label{cocycleid}
\sigma(g,h)\sigma(gh,k)=\sigma(g,hk)\sigma(h,k)\text{ for all $g,h,k\in G$.}
\end{equation}
The set $Z^2(G,Z)$ of all such cocycles is a group under pointwise operations from $Z$. Two cocycles $\sigma,\tau\in Z^2(G,Z)$ are \emph{equivalent} if there is a function $b:G\to Z$ such that $\sigma(g,h)=b(g)b(h)b(gh)^{-1}\tau(g,h)$ for all $g,h$. The cocycles which are equivalent to the trivial cocycle $\sigma\equiv 1$ form a subgroup of $Z^2(G,Z)$, and the quotient $H^2(G,Z)$ is called the \emph{$2$nd cohomology group}. When $Z=\T$, $Z^2(G,\T)$ is compact in the topology of pointwise convergence, and $H^2(G,\T)$ is a compact Hausdorff topological group.

The following are our key examples.

\begin{ex}\label{cocyclesonZ}
For each strictly upper triangular $k\times k$ matrix $A$ with entries in $[0,1)$, there is a $2$-cocycle $\sigma_A\in Z^2(\Z^k,\T)$ given by $\sigma_A(m,n)=\exp(2\pi i m^tAn)$, where $m^t$ denotes the transpose of the column vector $m$. Two such cocycles $\sigma_A$, $\sigma_B$ are equivalent if and only if $A=B$, and every cocycle in $Z^2(\Z^k,\T)$ is equivalent to one of these \cite{B}. The map $A\mapsto \sigma_A$ satisfies $\sigma_A\sigma_B=\sigma_{A+B}$, and induces an isomorphism of the torus $\T^{k(k-1)/2}$ onto $H^2(\Z^k,\T)$. Similarly, integer matrices $A$ also give cocycles $\tau_A\in Z^2(\Z^k,\Z)$ such that $\tau_A(m,n)=m^tAn$.
\end{ex}

Cocycles and the cohomology group $H^2(G,Z)$ are important in group theory because they give the central extensions of $G$ by $Z$: groups $H$ such that the centre $Z(H)$ contains a copy of $Z$ with $H/Z$ isomorphic to $G$. Given $\sigma\in Z^2(G,Z)$, a central extension is the group $H_\sigma$ with underlying set $Z\times G$ and product given by
\[
(z_1,g_1)(z_2,g_2)=(z_1 z_2\sigma(g_1,g_2),g_1g_2).
\]
Two such extensions $H_\sigma$ and $H_\tau$ are isomorphic (as extensions) if and only if $\sigma$ is equivalent to $\tau$. Conversely, given a central extension $H$ of $G$ by $Z$, we can take $c:G\to H$ to be any section for the quotient map $H\to G$ such that $c(e_G)=e_H$, and then 
\begin{equation}\label{cvssigma}\sigma(g_1,g_2)=c(g_1)c(g_2)c(g_1g_2)^{-1}
\end{equation}
defines a $2$-cocycle.

\begin{ex}
One might prefer to use a different realisation of the central extension: for example, for $\sigma\in Z^2(\Z^2,\Z)$ defined by $\sigma(m,n)=m_1n_2$, we would use the integer Heisenberg group $H(\Z)$ of upper triangular matrices
\[
\begin{pmatrix}
1&m_1&p\\0&1&m_2\\0&0&1
\end{pmatrix}.
\]
The cocycles on  $H(\Z)$ itself were computed in \cite{Pa} (see also \cite[Examples~1.4]{PR}).
\end{ex}

We now describe how to deform a Fell bundle.

\begin{prop}
Suppose that $p:B\to G$ is a Fell bundle over a discrete group $G$, and that $\sigma\in Z^2(G,\T)$ is a cocycle.  Then there is a Fell bundle $p:B(\sigma)\to G$ with the same underlying set $B$, the same projection $p$, the same Banach space structures on the fibres, and the other operations given in terms of those in $B$ by
\[
b\cdot_\sigma c=\sigma(g,h)bc\text{ and }b^{*_\sigma}=\overline{\sigma(g,g^{-1})}b^*\text{ for $b\in B_g$ and $c\in B_h$.}
\]
\end{prop}

\begin{proof}
The cocycle identity implies that $\cdot_\sigma$ is an associative multiplication. To see that $b\mapsto b^{*_\sigma}$ is an involution, we use the identity $\sigma(g,g^{-1})=\sigma(g^{-1},g)$, which follows from the cocycle identity for the triple $g,g^{-1},g$. Indeed, for $b\in B_g$ we have $b^*\in B_{g^{-1}}$, and hence
\begin{align*}
(b^{*_\sigma})^{*_\sigma}&=\big(\overline{\sigma(g,g^{-1})}b^*\big)^{*_\sigma}=\overline{\sigma(g^{-1},g)}\big(\overline{\sigma(g,g^{-1})}b^*\big)^*\\
&=\overline{\sigma(g^{-1},g)}\sigma(g,g^{-1})(b^*)^*=b.
\end{align*}
Similar computations using the cocycle identity show that $\cdot_\sigma$ and ${}^{*_\sigma}$ (indeed, the same calculations needed to show that the twisted group algebra $C_c(G,\sigma)$ is a $*$-algebra) show that $\cdot_\sigma$ and ${}^{*_\sigma}$  have the other required algebraic properties. Finally, for $b\in B_g$,
\[
b^{*_\sigma}\cdot_\sigma b=\overline{\sigma(g,g^{-1})}b^*\cdot_\sigma b=\overline{\sigma(g,g^{-1})}\sigma(g^{-1},g)b^*b=b^*b
\]
is positive in the $C^*$-algebra $B(\sigma)_e=B_e$.
\end{proof}

\begin{rmk}
Group algebras and crossed products in which the multiplication is twisted by a $2$-cocycle have been studied since the early days of crossed products \cite{Ze}. More recently, deformations of algebras carrying coactions of groups have been studied by Yamashita \cite{Y} and Bhowmick--Neshveyev--Sangha \cite{BNS}. Yamashita's deformations are of $C^*$-algebras carrying coactions of a discrete group, and hence by Quigg's theorem (as in \cite{Q} or Appendix~A) are equivalent to the $C^*$-algebras of our deformed Fell bundles --- indeed, Fell bundles themselves play a minor role in \cite[\S3]{Y}.  Bhowmick, Neshveyev and Sangha consider coactions of locally compact groups and deformations involving Borel cocyles; the examples in \cite[Examples 2.3]{LPRS} show that Fell-bundle techniques do not suffice in such situations. Both \cite{Y} and \cite{BNS} focus on applications to the Baum--Connes conjecture, and there seems to be relatively little overlap with the present paper.
\end{rmk}

\section{Deformations of graph algebras}\label{sec:galgs}

We now consider a graph $\Lambda$ of rank $k$, which we always assume is row-finite and has no sources. We use the standard coventions and notation of the subject. In particular for $\lambda,\mu\in \Lambda$, we write
\[
\Lambda^{\min}(\lambda,\mu)=\big\{(\alpha,\beta):\lambda\alpha=\mu\beta\text{ and }d(\lambda\alpha)=d(\lambda)\vee d(\mu)\big\},
\]
and $v\Lambda^n$ denotes the set of $\lambda\in \Lambda$ with $r(\lambda)=v$ (so $\lambda =v\lambda$) and $d(\lambda)=n$.

Our next result says that the graph algebra $C^*(\Lambda)=C^*(\{s_\lambda:\lambda\in \Lambda\})$ can be realised as the $C^*$-algebra of a Fell bundle in different ways. An important example of a functor to which the result applies is the degree functor $d:\Lambda\to\Z^k$. We will show at the end of the section that there are many others to which it applies, and our result is interesting when $k=1$.

\begin{prop}\label{labeltoFell}
Suppose that $\Lambda$ is a graph of rank $k$, and $\eta:\Lambda\to G$ is a functor from $\Lambda$ into a group $G$. For $g\in G$, we set
\begin{equation}\label{defBg}
B_g:=\clsp\{s_\lambda s_\mu^*: \eta(\lambda)\eta(\mu)^{-1}=g\},
\end{equation}
$B:=\bigsqcup_{g\in G}B_g$, and define $p:B\to G$ by $p(b)=g$ for  $b\in B_g$. Then with all the operations inherited from $C^*(\Lambda)$, $p:B\to G$ is a Fell bundle. The inclusion maps $\iota_g:B_g\to C^*(\Lambda)$ form a representation of $B$, and the corresponding homomorphism $\iota$ is an isomorphism of $C^*(B)$ onto $C^*(\Lambda)$.
\end{prop}

\begin{proof}
We first note that 
\[
(s_\lambda s_\mu^*)(s_\sigma s_\tau^*)=\sum_{(\alpha,\beta)\in\Lambda^{\min}(\mu,\sigma)}s_{\lambda\alpha}s_{\tau\beta}^*
\]
belongs to $B_{\eta(\lambda\alpha)\eta(\tau\beta)^{-1}}=B_{\eta(\lambda)\eta(\alpha)\eta(\beta)^{-1}\eta(\tau)^{-1}}$, which is the same as $B_{\eta(\lambda)\eta(\mu)^{-1}\eta(\sigma)\eta(\tau)^{-1}}$ because 
\begin{equation}\label{degpreserved}
\eta(\mu)\eta(\alpha)=\eta(\mu\alpha)=\eta(\sigma\beta)=\eta(\sigma)\eta(\beta).
\end{equation}
Thus $B_gB_h\subset B_{gh}$, and we also have $B_g^*\subset B_{g^{-1}}$. All the properties of the operations and norms follow from the corresponding properties in $C^*(\Lambda)$, and since $B_e$ is a $C^*$-subalgebra of $C^*(\Lambda)$, property \eqref{FBd} follows from spectral permanence:
\[
b^*b\geq 0\text{ in }C^*(\Lambda)\Longrightarrow \sigma_{C^*(\Lambda)}(b^*b)\subset [0,\infty)
\Longrightarrow \sigma_{B_e}(b^*b)\subset [0,\infty)\Longrightarrow 
b^*b\geq 0\text{ in }B_e.
\]

It is straightforward to check that the $\iota_g$ form a representation of $B$, and hence induce a homomorphism $\iota:C^*(B)\to C^*(\Lambda)$. It is surjective because the range is a $C^*$-subalgebra containing all the generators $s_\lambda$. If $\{\rho_g\}$ is another representation of $B$ in $C$, say, then the elements $\rho_{\eta(\lambda)}(s_\lambda)$ form a Cuntz--Krieger $\Lambda$ family in $C$, and there is a representation of $C^*(\Lambda)$ in $C$, so the corresponding representation $\rho:C^*(B)\to C$ factors through $\iota$. Thus $\iota$ is injective.
\end{proof}

Our first main theorem says that the $C^*$-algebra of the Fell bundle $B(\sigma)$ is generated by a family of partial isometries satisfying a twisted version of the orignal Cuntz--Krieger relations of \cite{KP}. Notice that only (CK2) is different:

\begin{thm}\label{twistgraphFell}
Suppose that $\Lambda$ is a row-finite $k$-graph with no sources, that $\eta:\Lambda\to G$ is a functor from $\Lambda$ into a group $G$, and that $B$ is the Fell bundle of Proposition~\ref{labeltoFell}, with fibres $B_g$ given by \eqref{defBg}. Suppose that $\sigma\in Z^2(G,\T)$ is a cocycle with values in the circle group $\T$. Let $\{s_\lambda:\lambda\in \Lambda\}$ be the canonical Cuntz--Krieger family in $C^*(\Lambda)$, viewed as a subset of $B$, and let $\pi^\sigma$ be the canonical representation of $B(\sigma)$ in $C^*(B(\sigma))$. Then the elements $t_\lambda:=\pi_{\eta(\lambda)}^\sigma(s_\lambda)$ of $C^*(B(\sigma))$ satisfy
\begin{enumerate}\item[]
\begin{enumerate}
\item[(CK1)] $\{q_v:=t_v:v\in \Lambda^0\}$ consists of mutually orthogonal projections;
\smallskip
\item[(CK2)] $t_\lambda t_\mu =\sigma(\eta(\lambda),\eta(\mu))t_{\lambda\mu}$ whenever $r(\mu)=s(\lambda)$;
\smallskip
\item[(CK3)] $t_\lambda^{*}t_\lambda=q_{s(\lambda)}$ for all $\lambda\in \Lambda$;
\smallskip
\item[(CK4)] $q_v=\sum_{\lambda\in v\Lambda^n} t_\lambda t_\lambda^{*}$ for all $v\in \Lambda^0$ and $n\in \N^k$.
\end{enumerate}
\end{enumerate}
Conversely, given a family $\{T_\lambda:\lambda\in \Lambda\}$ in a $C^*$-algebra $A$ satisfying \textnormal{(CK1--4)}, there is a homomorphism $\rho_T:C^*(B(\sigma))\to A$ such that $\rho_T(t_\lambda)=T_\lambda$ for all $\lambda\in \Lambda$.
\end{thm}

To prove that the elements $t_\lambda$ satisfy (CK1--4) we need to do some calculations in $B(\sigma)$.

\begin{lem}\label{compinbdle}
Suppose we have $\Lambda$, $\eta$, $B$ and $\sigma$ as in Theorem~\ref{twistgraphFell}. Let $\{s_\lambda:\lambda\in \Lambda\}$ be the canonical Cuntz--Krieger family in $C^*(\Lambda)$, viewed as a subset of $B(\sigma)$. Then
\begin{enumerate}\item[]
\begin{enumerate}
\item[(CK$\sigma$1)] $\{p_v:v\in \Lambda^0\}$ consists of mutually orthogonal projections (in other words, $p_v\cdot_\sigma p_v=p_v=p_v^{*_\sigma}$ and $p_v\cdot_\sigma p_w=\delta_{v,w}p_v$);
\smallskip
\item[(CK$\sigma$2)] $s_\lambda\cdot_\sigma s_\mu =\sigma(\eta(\lambda),\eta(\mu))s_{\lambda\mu}$ whenever $r(\mu)=s(\lambda)$;
\smallskip
\item[(CK$\sigma$3)] $s_\lambda^{*_\sigma}\cdot_\sigma s_\lambda=p_{s(\lambda)}$ for all $\lambda\in \Lambda$;
\smallskip
\item[(CK$\sigma$4)] $p_v=\sum_{\lambda\in v\Lambda^n} s_\lambda\cdot_\sigma s_\lambda^{*_\sigma}$ for all $v\in \Lambda^0$ and $n\in \N^k$.
\end{enumerate}
\end{enumerate}
\end{lem}

\begin{proof}
To see (CK$\sigma$1), we let $v,w\in \Lambda^0$. Then
\[
p_v\cdot_\sigma p_w=\sigma(\eta(v),\eta(w))p_vp_w=\sigma(e,e)p_vp_w=p_vp_w=\delta_{v,w}p_v,
\]
and similarly $p^{*\sigma}_v=p_v$. Since $s_{\lambda}\in B_{\eta(\lambda)}$ and $s_{\mu}\in B_{\eta(\mu)}$, (CK$\sigma$2) follows from the definition of the multiplication $\cdot_\sigma$. For (CK$\sigma$3), we use $\sigma(g,g^{-1})=\sigma(g^{-1},g)$ to compute
\begin{align*}
s_\lambda^{*_\sigma}\cdot_\sigma s_\lambda&=\overline{\sigma(\eta(\lambda),\eta(\lambda)^{-1})}s_\lambda^*\cdot_\sigma s_\lambda\\
&=\overline{\sigma(\eta(\lambda),\eta(\lambda)^{-1})}\,\sigma(\eta(\lambda)^{-1},\eta(\lambda))s_\lambda^*s_\lambda\\
&=s_\lambda^*s_\lambda=p_{s(\lambda)}.
\end{align*}
A similar computation shows that $s_\lambda\cdot_\sigma s_\lambda^{*_\sigma}=s_\lambda s_\lambda^*$, and hence (CK$\sigma$4) follows from the Cuntz--Krieger relation in $C^*(B)=C^*(\Lambda)$.
\end{proof}

Since the canonical maps $\pi_g^\sigma:B(\sigma)_g\to C^*(B(\sigma))$ form a representation of the Fell bundle $B(\sigma)$, Lemma~\ref{compinbdle} implies that the $t_\lambda$ satisfy (CK1--4).

As for the usual Cuntz--Krieger relations, the relations (CK1--4) have important consequences. First, (CK3) and (CK4) imply that
\begin{equation}\label{CKtoLeavitt}
q_{r(e)}t_e=t_e=t_eq_{s(e)}\text{ for every $e\in E^1$}.
\end{equation}
Next, we deduce from (CK1) and (CK4) that the projections $\{t_\lambda t_\lambda^{*}:d(\lambda)=n\}$ are mutually orthogonal. With this observation, we can strengthen (CK3) to 
\begin{equation}\label{CK3+}
t_\lambda^{*}t_{\mu}=\delta_{\lambda,\mu}q_{s(\lambda)}\text{ whenever $d(\mu)=d(\lambda)$}.
\end{equation} 
Now the usual calculation (for example, from \cite[Lemma~10.6]{R}), shows that for all $\lambda,\mu\in \Lambda$ and $n\geq d(\lambda)\vee d(\mu)$, we have
\begin{equation}\label{usual2}
t_\lambda^{*}t_\mu=
\sum_{\{\alpha,\beta\;:\;\lambda\alpha=\mu\beta,\;d(\lambda\alpha)=n\}}\overline{\sigma(\eta(\lambda),\eta(\alpha))}\sigma(\eta(\mu),\eta(\beta))t_{\alpha}t_\beta^{*}.
\end{equation}
Since the cocycle $\sigma$ is scalar-valued, \eqref{usual2} implies that the $t_\lambda t_\mu^*$ span a $*$-subalgebra of $C^*(B(\sigma))$, and since this subalgebra contains dense subsets of all the fibres $\pi^\sigma_g(B_g)$, it is dense in $C^*(B(\sigma))$. 

To see that the $t_\lambda$ have the universal property described in the last sentence of Theorem~\ref{twistgraphFell}, we use an analogue of the gauge action for $C^*(B(\sigma))$.

\begin{prop}\label{gaugeaction}
Suppose that the $t_\lambda$ are as described in Theorem~\ref{twistgraphFell}. Then there is a continuous action $\beta$ of $\T^k$ on $C^*(B(\sigma))$ such that
\begin{equation}\label{newgauge}
\beta_z\big(t_\lambda t_\mu^{*})=z^{d(\lambda)-d(\mu)}t_\lambda t_\mu^*.
\end{equation}
\end{prop}

\begin{proof}
We consider the usual gauge action $\gamma:\T^k\to \Aut C^*(\Lambda)$. Then because $\gamma_z(s_\lambda s_\mu^*)$ is a scalar multiple of $s_\lambda s_\mu^*$, we have $\gamma_z:B_g\to B_g$ for every $g\in G$. Since $\gamma_z$ is norm-preserving and $B_g$ has the norm inherited from $C^*(\Lambda)$, $\gamma_z|_{B_g}$ is a norm-preserving isomorphism of $B_g$ onto $B_g$. We aim to prove that the $\gamma_z|_{B_g}$ form an automorphism of the Fell bundle $B(\sigma)$.

We have to show that the $\gamma_z|_{B_g}$ respect the multiplication and involution in $B(\sigma)$, and since they are linear and isometric, it suffices to work with the elements $s_\lambda s_\mu^*$. Suppose that $s_\lambda s_\mu^*\in B_g$ and $s_\nu s_\tau^*\in B_h$ (so that $g=\eta(\lambda)\eta(\mu)^{-1}$ and $h=\eta(\nu)\eta(\tau)^{-1}$). Then
\begin{align*}
\gamma_z((s_\lambda s_\mu^*)\cdot_\sigma(s_\nu s_\tau^*))
&=\sigma(g,h)\gamma_z((s_\lambda s_\mu^*)(s_\nu s_\tau^*))\\
&=\sigma(g,h)\gamma_z\Big(\sum_{(\alpha,\beta)\in \Lambda^{\min}(\mu,\nu)}s_{\lambda\alpha}s_{\tau\beta}^*\Big)\\
&=\sigma(g,h)\sum_{(\alpha,\beta)\in \Lambda^{\min}(\mu,\nu)}z^{d(\lambda)+d(\alpha)-d(\tau)-d(\beta)}s_{\lambda\alpha}s_{\tau\beta}^*.
\end{align*}
For $(\alpha,\beta)\in \Lambda^{\min}(\mu,\nu)$ we have $d(\alpha)-d(\beta)=d(\nu)-d(\mu)$, so the coefficients inside the sum are all the same, and we can factor it out. Now we recognise the sum as $(s_\lambda s_\mu^*)(s_\nu s_\tau^*)$, and split the power of $z$ to get
\begin{align*}
\gamma_z((s_\lambda s_\mu^*)\cdot_{\sigma}(s_\nu s_\tau^*))
&=\sigma(g,h)(z^{d(\lambda)-d(\mu)}s_\lambda s_\mu^*)(z^{d(\nu)-d(\tau)}s_\nu s_\tau^*)\\
&=\gamma_z(s_\lambda s_\mu^*)\cdot_\sigma\gamma_z(s_\nu s_\tau^*).
\end{align*}
Thus the $\gamma_z|_{B_g}$ preserve the multiplication in $B(\sigma)$, and a similar but easier calculation shows that they preserve the adjoint. Thus they form an isomorphism of Fell bundles, as claimed.

Now the universal property of $(C^*(B(\sigma)),\pi_g)$ implies that there is an automorphism $\beta_z$ of $C^*(B(\sigma))$ such that 
\[
\beta_z\big(\pi_{\eta(\lambda)\eta(\mu)^{-1}}(s_\lambda s_\mu^*)\big)=z^{d(\lambda)-d(\mu)}\pi_{\eta(\lambda)\eta(\mu)^{-1}}(s_\lambda s_\mu^*).
\]
Since $t_\lambda t_\mu^*$ is a scalar multiple of $\pi_{\eta(\lambda)\eta(\mu)^{-1}}(s_\lambda s_\mu^*)$, the automorphism $\beta_z$ satisfies \eqref{newgauge}. The usual $\epsilon/3$ argument shows that $\beta$ is continuous.
\end{proof}

\begin{rmk}
At this stage we could complete the proof of Theorem~\ref{twistgraphFell} by recognising the relations (CK1--4) as a special case of those in \cite[Definition~5.2]{KPS2}, and applying the results of that paper (in particular the gauge-invariant uniqueness theorem \cite[Corollary~7.7]{KPS2}). However, one naturally wonders whether there is a direct argument that exploits the Fell-bundle structure. We provide one, and then describe the connection to \cite{KPS2} more precisely in Corollary~\ref{appltoKPS}.
\end{rmk}

The action $\beta:\T^k\to \Aut C^*(B(\sigma))$ gives a Fell bundle $C$ over $\Z^k$ with fibres
\[
C^*(B(\sigma))_n:=\clsp\{t_\lambda t_\mu^*:d(\lambda)-d(\mu)=n\}.
\]
Then the inclusion maps $\iota_n$ give a representation of the Fell bundle $C$ in $C^*(B(\sigma))$, and hence induce a homomorphism $\iota:C^*(C)\to C^*(B(\sigma))$. This homomorphism intertwines the dual action of $\T^k$ on $C^*(C)$ and the action $\beta$ on $C^*(B(\sigma))$, and $\iota_e$ is injective on $C_0$ (being an inclusion map). Thus $\iota$ is an isomorphism of $C^*(C)$ onto $C^*(B(\sigma)) $. 

Now suppose that $\{T_\lambda:\lambda\in \Lambda\}$ are elements of a $C^*$-algebra $A$ satisfying (CK1--4). We will find a representation of $C^*(B(\sigma))$ in $A$ by building a representation $\{\rho_n\}$ of the Fell bundle $C$. First we need to understand the core $C^*(B(\sigma))_0=C^*(B(\sigma))^\beta$.

\begin{lem}\label{matrixunits}
Suppose that $\beta:\T^k\to \Aut C^*(B(\sigma))$ is the action of Proposition~\ref{gaugeaction}. Then for each $v\in \Lambda^0$ and $n\in \N^k$, $\{t_\lambda t_\mu^*:\lambda,\mu\in \Lambda^nv\}$ is a set of nonzero matrix units in $C^*(B(\sigma))^\beta$.
\end{lem}

\begin{proof}
Suppose $\lambda,\mu,\nu,\tau\in \Lambda^nv$. We calculate, using first \eqref{CK3+} and then \eqref{CKtoLeavitt}:
\begin{equation}\label{matrix}
(t_\lambda t_\mu^*)(t_\nu t_\tau^*)=\delta_{\mu,\nu}t_\lambda q_{v} t_\tau^*=\delta_{\mu,\nu}t_\lambda t_\tau^*.
\end{equation}
Thus they are matrix units. Since the vertex projections belong to $B(\sigma)_e=B_e$, and the left regular representation of $B(\sigma)$ is faithful on $B_e$, the vertex projections $q_v$ are all nonzero. Now (CK3) implies that each $t_\lambda$  is nonzero, and so is every $t_\lambda t_\mu^*$.
\end{proof}

Recall that $\{T_\lambda:\lambda\in \Lambda\}\subset A$ satisfy (CK1--4). The calculation \eqref{matrix} implies that 
$\{T_\lambda T_\mu^*:\lambda,\mu\in \Lambda^nv\}$ is a set of matrix units in $A$, and hence by Lemma~\ref{matrixunits} there is a homomorphism $\phi_{n,v}$ of $D(\sigma)_{n,v}:=\clsp\{t_\lambda t_\mu^*:\lambda,\mu\in \Lambda^nv\}$ into $A$ such that $\phi_{n,v}(t_\lambda t_\mu^*)=T_\lambda T_\mu^*$. These combine to give a homomorphism $\phi_n$ of $D_n:=\bigoplus_{v\in \Lambda^0}D(\sigma)_{n,v}$ into $A$. The Cuntz--Krieger relation (CK4) implies that for $m\leq n$ in $\N^k$ we have $D_m\subset D_n$, and that $\phi_n|_{D_m}=\phi_m$. Since the $\phi_n$ are all norm-decreasing, we get a homomorphism $\rho_0$ of $B(\sigma)_0=\overline{\bigcup_{n\in \N^k}D_n}$ into $A$ such that $\rho_0(t_\lambda t_\mu^*)=T_\lambda T_\mu^*$ for all $\lambda,\mu$ with $d(\lambda)=d(\mu)$.

The next lemma follows from a calculation using the factorisation property and \eqref{CK3+}:

\begin{lem}\label{intocore}
Suppose that $\nu,\tau\in \Lambda$, and write $n:=d(\nu)-d(\tau)$ as $n=n_+-n_-$ with $n_{\pm}\in \N^k$ and $n_+\wedge n_-=0$. Then
\[
\lambda\in \Lambda^{n_+},\;\mu \in \Lambda^{n_-}\text{ and }t_\lambda^* t_\nu t_\tau^*t_\mu\not=0\Longrightarrow \lambda=\nu(0,n_+)\text{ and }\mu=\tau(0,n_-).
\]
\end{lem}

Lemma~\ref{intocore} implies that, for every $n\in \N^k$ and every $a$ in 
\[
C^*(B(\sigma))^0_n:=\lsp\{t_\nu t_\tau^*:d(\nu)-d(\tau)=n\},
\]
the sum
\[
\sum_{\lambda\in \Lambda^{n_+}}\sum_{\mu \in \Lambda^{n_-}}T_\lambda\rho_0(t_\lambda^* at_\mu)T_\mu^*
\]
has at most finitely many nonzero terms. Thus there are well-defined linear maps $\rho_n:C^*(B(\sigma))^0_n\to A$ such that
\[
\rho_n(a)=\sum_{\lambda\in \Lambda^{n_+}}\sum_{\mu \in \Lambda^{n_-}}T_\lambda\rho_0(t_\lambda^* a t_\mu)T_\mu^*.
\]
We aim to prove that each $\rho_n$ extends to a linear map on the closure $C^*(B(\sigma))_n$ of $C^*(B(\sigma))_n^0$, and that the $\rho_n$ form a representation of the Fell bundle $C$.

Our estimate of the norm of $\rho_n$ hinges on the following lemma.

\begin{lem}\label{pullTin}
Suppose $b\in C^*(B(\sigma))_0^0$ and $d(\lambda)=d(\mu)$. Then $T_\lambda\rho_0(b)T_\mu^*=\rho_0(t_\lambda b t_\mu^*)$.
\end{lem}

\begin{proof}
By linearity, it suffices to prove this for $b=t_\nu t_\tau^*$ with $d(\nu)=d(\tau)$. Then 
\begin{align*}
T_\lambda\rho_0(b)T_\mu^*&=T_\lambda(T_\nu T_\tau^*)T_\mu^*=\sigma(\eta(\lambda),\eta(\nu))T_{\lambda\nu}T_{\mu\tau}^*\overline{\sigma(\eta(\mu),\eta(\tau))}\\
&=\rho_0\big(\sigma(\eta(\lambda),\eta(\nu))t_{\lambda\nu}t_{\mu\tau}^*\overline{\sigma(\eta(\mu),\eta(\tau))}\big)\\
&=\rho_0(t_\lambda t_\nu t_\tau^*t_{\mu}^*)=\rho_0(t_\lambda b t_\tau^*).\qedhere
\end{align*}
\end{proof}

\begin{proof}[End of the proof of Theorem~\ref{twistgraphFell}]
We now fix $n\in \N^k$ and $a\in C^*(B(\sigma))^0_n$. Write $n=n_+-n_-$ with $n_{\pm}\in \N^k$ and $n_+\wedge n_-=0$. Then for every pair $\nu,\tau$ with $d(\nu)-d(\tau)=n$, we have $d(\nu)\geq n_+$ and $d(\tau)\geq n_-$, and we can factor $\nu=\lambda \nu'$ and $\tau=\mu\tau'$ for some $\lambda\in \Lambda^{n_+}$ and $\mu \in \Lambda^{n_-}$. Then each $t_\nu t_\tau^*$ is a scalar multiple of $t_\lambda t_{\nu'}t_{\tau'}^*t_\mu$, and hence we can rearrange $a$ as a finite sum
\[
a=\sum_{\lambda\in \Lambda^{n_+}}\sum_{\mu \in \Lambda^{n_-}}t_\lambda a_{\lambda,\mu}t_\mu^*
\]
in which $a_{\lambda,\mu}$ belongs to $q_{s(\lambda)}C^*(B(\sigma))^0_0q_{s(\mu)}$. The relations \eqref{CK3+} and \eqref{CKtoLeavitt} then imply that
\[
\rho_n(a)=\sum_{\lambda\in \Lambda^{n_+}}\sum_{\mu \in \Lambda^{n_-}}T_\lambda\rho_0(t_\lambda^*t_\lambda a_{\lambda,\mu}t_\mu^* t_\mu)T_\mu^*
=\sum_{\lambda\in \Lambda^{n_+}}\sum_{\mu \in \Lambda^{n_-}}T_\lambda\rho_0(a_{\lambda,\mu})T_\mu^*.
\]
Thus
\begin{align*}
\|\rho_n(a)\|^2&=\Big\|\sum_{\lambda,\mu}\sum_{\lambda',\mu'}T_\lambda\rho_0(a_{\lambda,\mu})T_{\mu}^*T_{\mu'}\rho_0(a_{\lambda',\mu'})^*T_{\lambda'}^*\Big\|\\
&=\Big\|\sum_{\lambda,\mu}\sum_{\lambda'}T_\lambda\rho_0(a_{\lambda,\mu}q_{s(\mu)}a_{\lambda',\mu}^*)T_{\lambda'}^*\Big\|,
\end{align*}
and hence Lemma~\ref{pullTin} implies that
\[
\|\rho_n(a)\|^2=\Big\|\sum_{\lambda,\mu,\lambda'}\rho_0(t_\lambda a_{\lambda,\mu}q_{s(\mu)}a_{\lambda',\mu}^*t_{\lambda'}^*)^*\Big\|.
\]
We now write
\[
q_{s(\mu)}=t_\mu^* t_\mu=\sum_{\mu'\in \Lambda^{n_-}}t_\mu^* t_{\mu'},
\]
and obtain
\[
\|\rho_n(a)\|^2
=\Big\|\sum_{\lambda,\mu}\sum_{\lambda',\mu'}\rho_0(t_\lambda a_{\lambda,\mu}t_\mu^*)\rho_0(t_{\mu'}a_{\lambda',\mu'}^*t_{\lambda'}^*)\Big\|
=\|\rho_0(aa^*)\|\leq \|a\|^2.
\]
We can now extend $\rho_n$ to a norm-decreasing linear map of $C^*(B(\sigma))_n$ into $A$.

We still have to check that the $\rho_n$ form a representation of $C$. Since the $\rho_n$ are continuous and linear, it suffices to check this on generators. Another application of \eqref{CK3+} shows that $\rho_n(t_\nu t_\tau^*)=T_\nu T_\tau^*$ whenever $d(\nu)-d(\tau)=n$. Suppose that $d(\lambda)-d(\mu)=m$ and $d(\nu)-d(\tau)=n$. Then
\begin{align*}
\rho_m(t_\lambda t_\mu^*)\rho_n(t_\nu t_\tau^*)&=(T_\lambda T_\mu^*)(T_\nu T_\tau^*)\\
&=\sum_{(\alpha,\beta)\in \Lambda^{\min}(\mu,\nu)}T_\lambda T_\alpha T_\beta^*T_\tau^*\\
&=\sum_{(\alpha,\beta)\in \Lambda^{\min}(\mu,\nu)}\sigma(\eta(\lambda),\eta(\alpha))T_{\lambda\alpha}T_{\tau\beta}^*\overline{\sigma(\eta(\tau),\eta(\beta))}\\
&=\rho_{m+n}\Big(\sum_{(\alpha,\beta)\in \Lambda^{\min}(\mu,\nu)}\sigma(\eta(\lambda),\eta(\alpha))t_{\lambda\alpha}t_{\tau\beta}^*\overline{\sigma(\eta(\tau),\eta(\beta))}\Big)\\
&=\rho_{m+n}\big((t_\lambda t_\mu^*)(t_\nu t_\tau^*)\big).
\end{align*}
The equation $\rho_n(t_\nu t_\tau^*)=T_\nu T_\tau^*$ immediately implies that $\rho_n(a)^*=\rho_{-n}(a^*)$. Thus $\{\rho_n\}$ is a representation of $C$ in $A$, and there is a homomorphism $\rho:C^*(C)\to A$ such that $\rho(t_\lambda)=T_\lambda$. Pulling this over to $C^*(B(\sigma))$ under the isomorphism $\iota:C^*(C)\to C^*(B(\sigma))$ gives a homomorphism $\rho:C^*(B(\sigma))\to A$ such that $\rho(t_\lambda)=T_\lambda$.

This completes the proof of Theorem~\ref{twistgraphFell}.
\end{proof}

We now show that the  $C^*$-algebras $C^*(B(\sigma))$ are ``twisted higher-rank graph $C^*$-algebras,'' as studied in \cite{KPS1, KPS2, KPS3}. Define $c_\sigma:\{(\lambda,\mu):s(\lambda)=r(\mu)\}\to \T$ by 
\begin{equation}\label{defcatcocycle}
c_\sigma(\lambda, \mu)=\sigma(\eta(\lambda),\eta(\mu)).
\end{equation} 
Then $c_\sigma$ is a ``categorical $2$-cocycle'', as in \cite[Lemma~3.8]{KPS2}, and the family $\{t_\lambda\}$ is a  $(\Lambda,c_\sigma)$-family in $C^*(B(\sigma))$ in the sense of \cite[Definition~5.2]{KPS2}. Theorem~\ref{twistgraphFell} says that $(C^*(B(\sigma)),\{t_\lambda\})$ is universal for $C^*$-algebras generated by such families. Hence we have:  

\begin{cor}\label{appltoKPS}
Suppose we have $\Lambda$, $\eta$, $B$ and $\sigma$ as in Theorem~\ref{twistgraphFell}.  Define $c_\sigma$ using \eqref{defcatcocycle}. Then there is an isomorphism of $C^*(\Lambda,c_\sigma)$ onto $C^*(B(\sigma))$ which carries the canonical Cuntz--Krieger $(\Lambda,c_\sigma)$-family into $\{t_\lambda\}=\{\pi_{\eta(\lambda)}^\sigma(s_\lambda)\}$.
\end{cor}

\begin{ex}
For any $k$-graph $\Lambda$, we can apply Theorem~\ref{twistgraphFell} to the degree functor $d:\Lambda\to \N^k\subset \Z^k$. Then $B$ is a Fell bundle over $\Z^k$ in which $B_0$ is the usual core $C^*(\Lambda)^\gamma$, and 
\[
B_n=\clsp\{s_\lambda s_\mu^*:d(\lambda)-d(\mu)=n\}.
\]
We can then deform $B$ using the cocycles $\sigma_A$ on $\Z^k$ described in Example~\ref{cocyclesonZ}.
\end{ex}

\begin{ex}\label{twistedgpalgs}
Suppose that $\Lambda$ is the unique $k$-graph with one vertex $v$ and $k$ loops $\lambda_i$ satisfying $d(\lambda_i)=e_i$, and $A=(a_{ij})$ is as in the preceding example. Then $p_v$ is an identity for $C^*(B(\sigma_A))$, and (CK3--4) say that the $t_{\lambda_i}$ are unitary. The relation (CK2) then reduces to
\begin{equation}\label{desnoncommT}
t_{\lambda_i}t_{\lambda_j}=\exp(2\pi ia_{ij})t_{\lambda_i\lambda_j}=\exp(2\pi ia_{ij})t_{\lambda_j}t_{\lambda_i}\text{ for $i<j$.}
\end{equation}
Thus the universal property of $C^*(B(\sigma_A))=C^*(t_{\lambda_i})$ is that of the twisted group $C^*$-algebra $C^*(\Z^k,\sigma_A)$, which is commonly known as a ``noncommutative torus''. (This is slightly different from the description in \cite[Example~7.9]{KPS1}, but only because we are  parametrising $H^2(\Z^k,\T)$ by upper triangular matrices rather than skew-symmetric ones. With our convention, \eqref{desnoncommT} reduces to the usual commutation relation $t_{\lambda_1}t_{\lambda_2}=e^{2\pi i\theta}t_{\lambda_2}t_{\lambda_1}$ when $k=2$.) The other main examples in \cite{KPS1} can also be obtained from the degree functor and cocycles on $\Z^2$.
\end{ex}

\begin{ex}
For every $k$-graph $\Lambda$ and $v\in \Lambda^0$, there is a functor $\eta$ of $\Lambda$ into the fundamental group $\pi_1(\Lambda,v)$ such that the skew-product $k$-graph $\Lambda\times_\eta \pi_1(\Lambda,v)$ is a universal covering of $\Lambda$ (see \cite[Corollary~6.5]{PQR}). So with $B$ as in Proposition~
\ref{labeltoFell}, we can deform $C^*(\Lambda)$ by cocycles on $\pi_1(\Lambda,v)$. However, $\pi_1(\Lambda,v)$  may not have interesting cocycles. For example, if $E$ is the $1$-graph with one vertex $v$ and two edges $e,f$, then $\pi_1(E,v)$ is the free group on $E^1=\{e,f\}$, and hence has no nontrivial cocycles (because every central extension of a free group splits). But a functor into $G$ gives functors into every quotient of $G$, and these quotients can have nontrivial cocycles. For example, with $E$ as above, there is a functor $\eta:E^*\to \Z^2$ such that $\eta(e)=(1,0)$ and $\eta(f)=(0,1)$, and then we can deform the associated Fell bundle using cocycles on $\Z^2$. 
\end{ex}

\section{Uniqueness theorems and amenability}\label{sec:unique}

Here we prove versions of the two main uniqueness theorems for graph algebras. Our gauge-invariant uniqueness theorem is slightly stronger than the usual one \cite[Theorem~3.4]{KP}, in that our continuity hypothesis on the action $\beta$ is a bit weaker. We use the extra strength in our proof of amenability of the deformed Fell bundles (Corollary~\ref{amenability}).

\begin{thm}[The gauge-invariant uniqueness theorem]\label{giut}
Suppose that $\Lambda$ is a graph of rank $k$, $\eta:\Lambda\to G$ is a functor from $\Lambda$ into a group $G$, and $B$ is the Fell bundle of Proposition~\ref{labeltoFell}, with fibres $B_g$ given by \eqref{defBg}. Suppose that $\sigma\in Z^2(G,\T)$ is a cocycle. Suppose there is a representation $\rho$ of $B(\sigma)$ in a multiplier algebra $M(A)$, and that there is a strictly continuous action $\alpha$ of $\T^k$ on $M(A)$ such that $\rho_g\circ\beta_z=\alpha_z\circ \rho_g$ for all $z\in \T^k$ and $g\in G$. If $\rho_e(p_v)\not=0$ for every $v\in \Lambda^0$, then the corresponding representation $\rho$ of $C^*(B(\sigma))$ is faithful.
\end{thm}

\begin{proof}
For each $v\in \Lambda^0$ and $n\in \Z^k$, we know from Lemma~\ref{matrixunits} that $\{t_\lambda t_\mu^*:\lambda,\mu\in \Lambda^nv\}$ is a set of nonzero matrix units.
Since the projections $\rho_e(p_v)$ are all nonzero, $\rho_e$ is faithful on each 
\[
B_{v,n}:=\clsp\{t_\lambda t_\mu^*:\lambda,\mu\in \Lambda^nv\} \cong \KK(\ell^2(\Lambda^nv)).
\]
We have $B_{v,n}B_{w,n}=0$ for $v\not= w$, so $B_n:=\clsp\{t_\lambda t_\mu^*:\lambda,\mu\in \Lambda^n\}$
is the direct sum of the subalgebras $B_{v,n}$, and $\rho_e$ is faithful on $B_n$. We deduce that $\rho_e$ is faithful on $C^*(B(\sigma))^\beta=\overline{\bigcup_{n\in \N^k}B_n}$.

We now want to run the standard argument, as in \cite[pages~28--29]{R}, for example. However, we need to do things in the right order because the action is only strictly continuous (see the footnote on page~278 of \cite{tfb}). As usual, we consider the expectation $\Phi$ of $C^*(B(\sigma))$ onto $C^*(B(\sigma))^\beta$ given by $\Phi(a)=\int_{\T^k} \beta_z(a)\,dz$. We choose a faithful nondegenerate representation of $C^*(B(\sigma))$ as operators on a Hilbert space $H$; now the map $z\mapsto \gamma_z(a)h$ is norm-continuous for each $h\in H$. Next we take $a\in C^*(B(\sigma))$ and $h,k\in H$, and calculate, using the $C^*$-algebra-valued integrals in \cite[Lemmas~C.3 and~C.11]{tfb}: 
\begin{align*}
\big|\big(\rho(\Phi(a))h\,|\,k\big)\big|
&=\Big|\Big(\rho\Big(\int_{\T^k} \beta_z(a)\,dz\Big)h\;|\;k\Big)\Big|
=\Big|\int_{\T^k}\big(\rho(\beta_z(a))h\;|\;k\big)\,dz\Big|\\
&\leq \int_{\T^k}\big|\big(\rho(\beta_z(a))h\;|\;k\big)\big|\,dz
= \int_{\T^k}\big|\big(\alpha_z(\rho(a))h\;|\;k\big)\big|\,dz.
\end{align*}
For each $z$, we have
\[
\big|\big(\alpha_z(\rho(a))h\;|\;k\big)\big|\leq \|\alpha_z(\rho(a))\|\,\|h\|\,\|k\|
=\|\rho(a)\|\,\|h\|\,\|k\|.
\]
Thus 
\[
\big|\big(\rho(\Phi(a))h\,|\,k\big)\big|\leq \|\rho(a)\|\,\|h\|\,\|k\| \text{ for all $h,k\in H$,}
\]
and we have $\|\rho(\Phi(a))\|\leq \|\rho(a)\|$. With this inequality, we can return to the standard argument (as in the display (3.11) in \cite{R}, for example.)
\end{proof}

We will show as a corollary of Theorem~\ref{giut} that the bundles $B(\sigma)$ are amenable. For this, we need the following standard lemma, which is implicitly contained in \cite[\S1]{C} and \cite[Proposition~2.27]{EKQR}, for example.

\begin{lem}\label{indaction}
Suppose that $A$ and $C$ are $C^*$-algebras and ${}_AX_C$ is an imprimitivity bimodule. Suppose we have a continuous action $\beta$ of a locally compact group $G$ on $C$, and a homomorphism $u$ of $G$ into the group of invertible linear transformations from $X$ to $X$ such that $u_g(x\cdot c)=u_g(x)\cdot c$, $\langle u_g(x),u_g(y)\rangle_C=\beta_g(\langle x,y\rangle_C)$ and $g\mapsto u_g(x)$ is continuous for each $x\in X$. Then there is a strictly continuous action $\alpha$ of $G$ on $M(A)$ such that $\alpha_g(m)\cdot x =u_g(m\cdot u_g^{-1}(x))$.
\end{lem}

\begin{cor}\label{amenability}
Suppose we have $\Lambda$, $\eta$, $B$ and $\sigma$ as in Theorem~\ref{giut}. Then the Fell bundle $B(\sigma)$ is amenable in the sense of Exel \cite{E}.
\end{cor}

\begin{proof}
We saw in the proof of Proposition~\ref{gaugeaction} that the gauge action $\gamma_z$ on $C^*(\Lambda)$ maps each $B_g$ isometrically into itself, and that the gauge action respects the structure of the twisted bundle $B(\sigma)$. Thus the $u^g_z:=\gamma_z|_{B_g}$ combine to give a linear automorphism $u_z$ of $\ell^2(B)$ which is compatible in the sense of Lemma~\ref{indaction} with the restriction $\beta_z|_{B_e}$. Thus there is a strictly continuous action $\alpha$ of $\T^k$ on $M(\KK(\ell^2(B)))=\LL(\ell^2(B))$ such that $\alpha_z(m)\cdot b= u_z(m\cdot u_{z^{-1}}(b))$ for $b\in B_g$. Now we take $m=s_\lambda s_\mu^*$ and $b=s_\nu s_\tau^*$ and check that the left regular representation $\lambda$ satisfies 
\begin{align*}
(\alpha_z(\lambda_{\eta(\lambda)\eta(\mu)^{-1}}s_\lambda s_\mu^*))\cdot b&=u_z((s_\lambda s_\mu^*)\cdot_\sigma u_z(s_\nu s_\tau^*))\\
&=z^{d(\lambda)-d(\mu)+d(\nu)-d(\tau)}((s_\lambda s_\mu^*)\cdot_\sigma z^{-(d(\nu)-d(\tau))}(s_\nu s_\tau^*))\\
&=(z^{d(\lambda)-d(\mu)}s_\lambda s_\mu^*)\cdot_\sigma (s_\nu s_\tau^*),
\end{align*}
which implies that $\alpha_z(\lambda_{g}(a))=\lambda_g(\beta_z(a))$. Since the action of each vertex projection $p_v$ on the summand $B_e$ of $\ell^2(B)$ is nonzero, we have $\lambda_e(p_v)\not=0$. Thus Theorem~\ref{giut} implies that the associated representation $\lambda$ of $C^*(B(\sigma))$ in $\LL(\ell^2(B))$ is injective. Thus it is an isomorphism of $C^*(B(\sigma))$ onto $C^*_r(B(\sigma))$, which means that $B(\sigma)$ is amenable.
\end{proof}

We now give a proof of Cuntz--Krieger uniqueness. Suppose that $\Lambda$ is a row-finite $k$-graph  with no sources, and recall Robertson and Sims' finite-path formulation of aperiodicity \cite[Lemma~3.2]{RS}: $\Lambda$ is \emph{aperiodic} if for every $v\in \Lambda^0$ and every $m,n\in\N^k$ with $m\not= n$, there exists $\lambda\in v\Lambda$ such that $d(\lambda)\geq m\vee n$ and
\[
\lambda(m,m+d(\lambda)-(m\vee n))\not=\lambda(n,n+d(\lambda)-(m\vee n)).
\]

\begin{thm}[The Cuntz--Krieger uniqueness theorem]\label{CKthm}
Suppose that $\Lambda$ is a row-finite $k$ graph with no sources, and that $\Lambda$ is aperiodic. Suppose that $\eta:\Lambda\to G$ is a functor and $\sigma\in Z^2(G,\T)$, and $B$ is the Fell bundle of Proposition~\ref{labeltoFell}. Suppose that $\{T_\lambda:\lambda\in \Lambda\}$ are elements of a $C^*$-algebra $A$ satisfying (CK1--4), and that each $Q_v:=T_v$ is nonzero. Then the corresponding homomorphism $\pi_T:C^*(B(\sigma))\to A$ is injective.
\end{thm}

\begin{proof}
We follow the proof of the usual Cuntz--Krieger uniqueness theorem in \cite[\S6]{HRSW}. The crux of that proof is the following inequality for finite sums, which is Proposition~6.4 in \cite{HRSW}:
\begin{equation}\label{keyest}
\Big\|\sum_{\mu,\nu\in F}c_{\mu,\nu}T_\mu T_\nu^*\Big\|\geq \Big\|\sum_{\mu,\nu\in F,\;d(\mu)=d(\nu)}c_{\mu,\nu}T_\mu T_\nu^*\Big\|.
\end{equation}
We follow the proof of that proposition to the point where we define projections $Q_v$. For the current proof, we take 
\begin{equation*}
Q_v:=\sum_{\alpha\in Gv,\;d(\alpha)=n} T_{\alpha}T_{\lambda_v}T_{\lambda_v}^*T_\alpha^*\quad\text{rather than}\quad \sum_{\alpha\in Gv,\;d(\alpha)=n} T_{\alpha\lambda_v}T_{\alpha\lambda_v}^*.
\end{equation*}
(In the situation of \cite{HRSW}, the two would be the same, but here we would encounter extra cocycles in the calculations.) With this change, the rest of the proof of \cite[Proposition~6.4]{HRSW} applies almost verbatim in our situation.

Now we have \eqref{keyest}, the proof of \cite[Theorem~6.1]{HRSW} carries over.
\end{proof}

Next, recall that $\Lambda$ is \emph{cofinal} if for every pair $v,w\in \Lambda^0$, there exists $n\in \N^k$ such that $v\Lambda s(\lambda)\not=\emptyset$ for every $\lambda\in w\Lambda^n$.

\begin{cor}\label{simplicity}
Suppose that $\Lambda$ is a row-finite $k$-graph without sources, and that $\Lambda$ is both aperiodic and cofinal. Suppose that $\eta$, $\sigma$ and $B$ are as in Theorem~\ref{CKthm}. Then $C^*(B(\sigma))$ is simple.
\end{cor}

\begin{proof}
We will show that every nonzero homomorphism out of $C^*(B(\sigma))$ is injective. So suppose $\pi:C^*(B(\sigma))\to A$ is nonzero. We claim that every $\pi(q_v)$ is nonzero. So fix $v\in \Lambda^0$. Since the $t_\mu$ generate $C^*(B(\sigma))$, there is at least one $\mu\in \Lambda$ such that $\pi(t_\mu)\not=0$, and then $\pi(q_{s(\mu)})=\pi(t_\mu^*t_\mu)\not=0$ also. Taking $w=s(\mu)$ in the definition of cofinal gives $n$ such that $v\Lambda s(\lambda)$ is nonempty for every $\lambda\in s(\mu)\Lambda^n$. The relation (CK4) at $s(\mu)$ implies that at least one $\pi(t_{\lambda})$ is nonzero. Now any $\nu \in v\Lambda s(\lambda)$ has $\pi(t_\nu)\not=0$, and hence $\pi(q_v)=\pi(q_{r(\nu)})\geq \pi (t_\nu t_{\nu^*})$ is nonzero, as claimed. Now Theorem~\ref{CKthm} implies that $\pi$ is injective.
\end{proof}

\begin{rmk}
For the graph whose skeleton consists of $k$ loops at a single vertex, we saw in Example~\ref{twistedgpalgs} that the twisted graph algebras associated to the degree functor are the twisted group algebras $C^*(\Z^k,\sigma_A)$. This graph has only one path of each degree, and hence is not aperiodic, so Corollary~\ref{simplicity} does not apply. However, $C^*(\Z^k,\sigma_A)$ is often simple. Indeed, a theorem of Slawny says that whenever $\sigma_A$ is \emph{totally skew} in the sense that $m\in \Z^k$ and $\sigma_A(m,n)=\sigma_A(n,m)$ for all $n\in \Z^k$ imply $m=0$, the twisted group algebra is simple \cite[Theorem~3.7]{S}. (Slawny's result was extended to locally compact groups by Kleppner \cite{K}, and subsequently extended further by Green in \cite[Proposition~32]{G}.) The cocycles $\sigma_A$ are generically totally skew, and in particular whenever the entries $a_{ij}$ of $A$ are independent over the rationals. 
\end{rmk}

\section{Continuous bundles from deformed Fell bundles}\label{sec:field}

There are several theorems in the literature which say, roughly, that the various twisted crossed products associated to continuously varying cocycles are the fibres of a continuous $C^*$-bundle \cite{AP, Rie, W}. We next prove such a theorem in the context of Fell bundles.  For the twisted graph algebras of Kumjian, Pask and Sims \cite{KPS2}, a similar bundle was decribed in \cite{KPS3} (see Remark~\ref{KPSex}). However, our more general result suggests that Fell bundles provide an excellent environment for such theorems (see Remark~\ref{repgp}).

\begin{thm}\label{existbdle1}
Suppose that $p:B\to G$ is a Fell bundle over an amenable group and $\sigma\in Z^2(G,Z)$ is a cocycle with values in an abelian group $Z$. Then there is a continuous $C^*$-bundle $E$ over $\widehat Z$ with fibres $E_\gamma=C^*(B(\gamma\circ\sigma))$.
\end{thm}

To prove such a theorem, one usually finds a candidate $A$ for the $C^*$-algebra $\Gamma(E)$ of continuous sections of $E$, proves that $A$ is a $C(\widehat Z)$-algebra, and then identifies the quotients $A(\gamma)$ with the desired fibres $C^*(B(\gamma\circ \sigma))$. Then a theorem of Fell says there is an upper-semicontinuous $C^*$-bundle $E$ as claimed (see \cite[Theorem~C.25]{tfb2}, for example). To prove the bundle is continuous requires further argument, and this is where the amenability comes in (as it did in \cite{Rie} and \cite{W}).

Our candidate $A$ will be the $C^*$-algebra of a Fell bundle over a central extension $H$ of $G$ defined by the cocycle $\sigma$.  So we suppose that $q:H\to G$ is a central extension of $G$ by $Z$, and that $c:G\to H$ is a section for $q$ satisfying $c(e_G)=e_H$ such that the cocycle $\sigma$ is given by \eqref{cvssigma}.

We now form the pullback bundle $q^*B=\{(h,b):b\in B_{q(h)}\}$. For each $h$, the map $(h,b)\mapsto b$ is a bijection of $(q^*B)_h$ onto $B_{q(h)}$, and we use it to pull back the Banach space structure from $B_{q(h)}$. Then with $q^*p:(h,b)\mapsto h$, $q^*p:q^*B\to H$ is a Banach bundle over $H$. We claim that $q^*p:q^*B\to H$ is a Fell bundle over $H$ with 
\[
(h,b)(k,c)=(hk, bc)\text{ and }(h,b)^*=(h^{-1},b^*).
\]
Indeed, the algebraic properties \eqref{FBa}, \eqref{FBb} and \eqref{FBc} all follow from the corresponding properties of $B$. Then $(e,b)\mapsto b$ is a $C^*$-algebra isomorphism of $(q^*B)_{e_H}$ onto $B_{e_G}$, and therefore preserves positivity. Thus
\[
(h,b)^*(h,b)=(h^{-1},b^*)(h,b)=(e,b^*b)
\]
is positive in $(q^*B)_{e}$, and $q^*B$ also satisfies \eqref{FBd}.

Our candidate for the section algebra of the bundle $E$ in Theorem~\ref{existbdle1} is $C^*(q^*B)$. So next we have to show that $C^*(q^*B)$ is a $C(\widehat Z)$-algebra.

\begin{lem}\label{intocentre}
There is a unitary representation $v$ of $Z$ in the centre of $M(C^*(q^*B))$ such that
\begin{equation}\label{defv}
v_z\pi_h(h,b)=\pi_{zh}(zh,b)\text{ for all $(h,b)\in q^*B$.}
\end{equation}
\end{lem}

\begin{proof}
We choose a faithful nondegenerate representation $\rho$ of $C^*(q^*B)$ on a Hilbert space $\HH$. Then $\rho_e$ is a nondegenerate representation of $(q^*B)_e=\{e\}\times B_e$ on $\HH$. For $z\in Z$, $b\in B_e$ and $\xi\in \HH$ we have
\begin{align*}
\big(\rho_z(z,b)\xi\,|\,\rho_z(z,b)\xi\big)&=\big(\rho_{z^{-1}}(z^{-1},b)^*\rho_z(z,b)\xi\,|\,\xi\big)\\
&=\big(\rho_e(e,b^*b)\xi\,|\,\xi\big)\\
&=\big(\rho_e(e,b)\xi\,|\,\rho_e(e,b)\xi\big),
\end{align*}
and hence there is an isometry $V_z$ on $\HH$ such that $V_z(\rho_e(e,b)\xi)=\rho_z(z,b)\xi$.  This formula shows that $V_z$ multiplies $\rho(C^*(q^*B))$, and hence there is a unique unitary multiplier $v_z\in M(C^*(q^*B))$ such that $\bar\rho(v_z)=V_z$; a similar calculation shows that $v_z$ satisfies \eqref{defv}. It is invertible because $v_{z^{-1}}$ is an inverse. Finally, to see that $v_z$ is in the centre of $M(C^*(q^*B))$, we take $(h,b)\in q^*B$ and compute
\begin{align*}
(\pi_h(h,b)v_z)^*&=v_z^*\pi_h(h,b)^*=v_{z^{-1}}\pi_{h^{-1}}(h^{-1},b^*)=\pi_{z^{-1}h^{-1}}(z^{-1}h^{-1},b^*)\\
&=\pi_{hz}(hz,b)^*=(v_z\pi_h(h,b))^*,
\end{align*}
which implies that $\pi_h(h,b)v_z=v_z\pi_h(h,b)$.
\end{proof}

The integrated form $\pi_v$ of the representation $v$ of Lemma~\ref{intocentre} gives a unital homomorphism of $C^*(Z)$ into the centre of $M(C^*(q^*B))$. Since the Fourier transform is an isomorphism of $C^*(Z)$ onto $C(\widehat Z)$, this gives $C^*(q^*B)$ the structure of a $C(\widehat Z)$-algebra. Following the proof of \cite[Theorem~C.26]{tfb2} shows that there is an upper-semicontinuous bundle $E$ over $\widehat Z$ and a canonical isomorphism of $C^*(q^*B)$ onto $\Gamma(E)$. The proof shows that the fibre $E_\gamma$ is the quotient $C^*(q^*B)(\gamma)$ of $C^*(q^*B)$ by the ideal $I_\gamma:=\overline{(\ker\gamma)\cdot C^*(q^*B)}$, where $\ker\gamma$ is the kernel in $C^*(Z)$ of the integrated form of the representation $\gamma:Z\to \T=U(\C)$, and that the isomorphism takes $a\in C^*(q^*B)$ to the section $\gamma\mapsto a(\gamma)$.

In the next lemma we identify the fibre $E_\gamma$ with $C^*(B(\gamma\circ\sigma))$.

\begin{lem}\label{idfibres}
For $\gamma\in \widehat Z$, we denote by $\pi^\gamma$ the canonical representation of $B(\gamma\circ\sigma)$ in $C^*(B(\gamma\circ\sigma))$, and define $\rho^\gamma:q^*B\to C^*(B(\gamma\circ\sigma))$ by
\[
\rho^\gamma(h,b)=\gamma(hc(q(h))^{-1})\pi^\gamma_{q(h)}(b).
\]
Then $\rho^\gamma$ is a representation of $q^*B$. The (extension to the multiplier algebra of) the integrated form $\rho^\gamma:C^*(q^*B)\to C^*(B(\gamma\circ\sigma))$ then satisfies $\rho^\gamma(\pi_v(x))=\gamma(x)$ for $x\in C^*(Z)$, and induces an isomorphism of $C^*(q^*B)(\gamma):=C^*(q^*B)/I_\gamma$ onto $C^*(B(\gamma\circ\sigma))$.
\end{lem}

\begin{proof}
To see that $\rho^\gamma$ is a representation, we need to do some calculations. Because they involve many terms of the form $c(q(h))$, and no ambiguity seems likely, we write $c(h):=c(q(h))$.  We have
\begin{align*}
\rho^\gamma(h,b)\rho^\gamma(k,c)&=\gamma(hc(h)^{-1})\pi_{q(h)}^\gamma(b)\gamma(kc(k)^{-1})\pi_{q(k)}^\gamma(c)\\
&=\gamma(hc(h)^{-1}kc(k)^{-1})\pi_{q(hk)}^\gamma(b\cdot_{\gamma\circ\sigma}c)\\
&=\gamma(hkc(k)^{-1}c(h)^{-1})\pi_{q(hk)}^\gamma(\gamma\circ\sigma(q(h),q(k))bc)\quad\text{(since $kc(k)^{-1}\in Z$)}\\
&=\gamma(hkc(k)^{-1}c(h)^{-1}\sigma(q(h),q(k)))\pi_{q(hk)}^\gamma(bc)\\
&=\gamma(hkc(k)^{-1}c(h)^{-1}c(h)c(k)c(hk)^{-1})\pi_{q(hk)}^\gamma(bc)\\
&=\gamma(hkc(hk)^{-1})\pi_{q(hk)}^\gamma(bc)\\
&=\rho^\gamma(hk,bc)=\rho^\gamma((h,b)(k,c)).
\end{align*}
Similarly,
\begin{align*}
\rho^\gamma(h,b)^*
&=\big(\gamma(hc(h)^{-1})\pi_{q(h)}^\gamma(b)\big)^*=\overline{\gamma(hc(h)^{-1})}\pi_{q(h)}^\gamma(b)^*\\
&=\overline{\gamma(hc(h)^{-1})}\pi_{q(h)}^\gamma(b^{*_{\gamma\circ\sigma}})
=\overline{\gamma(hc(h)^{-1})}\pi_{q(h^{-1})}^\gamma\big(\overline{\gamma\circ\sigma(q(h),q(h^{-1}))}b^*\big)\\
&=\overline{\gamma(hc(h)^{-1}c(h)c(h^{-1})c(e)^{-1})}\pi_{q(h^{-1})}^\gamma(b^*)
=\overline{\gamma(c(h)c(h^{-1})hc(h)^{-1})}\pi_{q(h^{-1})}^\gamma(b^*)\\
&=\overline{\gamma(c(h^{-1})h)}\pi_{q(h^{-1})}^\gamma(b^*)\quad\text{(because $Z$ is central in $H$)}\\
&=\gamma(h^{-1}c(h^{-1})^{-1})\pi_{q(h^{-1})}^\gamma(b^*)=\rho^\gamma(h^{-1},b^*)=\rho((h,b)^*).
\end{align*}
For $(e,b)\in (q^*B)_e=B_e$, we have
\[
\rho^\gamma(e,b)=\gamma(ec(e)^{-1})\pi_{q(e)}^\gamma(b)=\pi_e^\gamma(b),
\]
and since $\pi_e$ is nondegenerate, so is $\rho^\gamma_e$. Thus $\rho^\gamma$ is a representation of $q^*B$.

For the next part, we write $\overline{\rho^\gamma}$ for the extension to $M(C^*(q^*B))$, which is characterised by  $\overline{\rho^\gamma}(m)\rho^\gamma(a)=\rho^\gamma(ma)$ for $m\in M(C^*(q^*B))$ and $a\in C^*(q^*B)$. It suffices to show that the two multipliers $\overline{\rho^\gamma}(v_z)$ and $\gamma(z)1_{M(C^*(q^*B))}$ agree on spanning elements $\pi_h(h,b)$. We have
\begin{align*}
\overline{\rho^\gamma}(v_z)\rho^\gamma(\pi_h(h,b))&=\rho^\gamma(v_z\pi_h(h,b))=\rho^\gamma(\pi_h(zh,b))\\
&=\gamma(zhc(zh)^{-1})\pi^\gamma_{q(zh)}(b)=\gamma(z)\gamma(hc(h)^{-1})\pi^\gamma_{q(h)}(b)\\
&=\gamma(z)\rho^\gamma(\pi_h(h,b)).
\end{align*}
Thus $\overline{\rho^\gamma}(v_z)=\gamma(z)1$ for all $z\in Z$, and it follows that $\overline{\rho^\gamma}(x)$ is the integrated form $\gamma(x)$ for all $x\in C^*(Z)$.

Since $\overline{\rho^\gamma}(x)=\gamma(x)$, $\rho^\gamma$ vanishes on the ideal $I_\gamma:=\overline{(\ker \gamma)\cdot C^*(q^*B)}$, and hence factors through a homomorphism $\phi_\gamma:C^*(q^*B)(\gamma)\to C^*(B(\gamma\circ \sigma))$. Since the range of $\rho^\gamma$ contains every $b$ in every $B_g=B(\gamma\circ\sigma)_g$, $\rho^\gamma$ and $\phi_\gamma$ are surjective. 

We prove that $\phi_\gamma$ is injective by constructing a left inverse for $\phi_\gamma$, using the universal property of $C^*(B(\gamma\circ\sigma))$ (and borrowing an idea from Anderson and Paschke \cite[Theorem~1.1]{AP}). We write $Q$ for the quotient map of $C^*(q^*B)$ onto $C^*(q^*B)(\gamma)$. Since $v_z-\gamma(z)1$ belongs to the kernel of $\gamma:C^*(Z)\to \C$, we have
\[
\pi_{zh}(zh,b)-\gamma(z)\pi_h(h,b)=((v_z-\gamma(z)1)\pi_h(h,b)\in I_\gamma=\overline{(\ker \gamma)\cdot C^*(q^*B)}.
\]
Thus we have
\begin{equation}\label{propQ}
Q(\pi_{zh}(zh,b))=\gamma(z)Q(\pi_h(h,b))=Q(\pi_h(h,\gamma(z)b)) \text{ for all $(h,b)\in q^*B$.}
\end{equation}
We now define $\tau_g:B_g\to C^*(q^*B)(\gamma)$ by $\tau_g(b)=Q\circ\pi(c(g),b):=Q(\pi_{c(g)}(c(g),b))$, and aim to prove that $\tau=\{\tau_g:g\in G\}$ is a representation of $B(\gamma\circ\sigma)$.

For $b\in B_{g_1}$ and $c\in B_{g_2}$, we calculate using \eqref{propQ}:
\begin{align*}
\tau_{g_1}(b)\tau_{g_2}(c)
&=Q\circ\pi(c(g_1)c(g_2),bc)=Q\circ\pi(\sigma(g_1,g_2)c(g_1g_2),bc)\\
&=Q\circ\pi(c(g_1g_2),\gamma(\sigma(g_1,g_2))bc)=\tau_{g_1g_2}(b\cdot_{\gamma\circ\sigma}c).
\end{align*}
Next, we let $b\in B_g$, and compute the adjoint. Here we observe that, since $Z$ is central, $c(g^{-1})c(g)$ belongs to $Z$, and hence we can  use \eqref{propQ} again:
\begin{align*}
\tau_g(b)^*&=Q\circ \pi((c(g),b)^*)=Q\circ \pi(c(g)^{-1},b^*)\\
&=Q\circ \pi(c(g)^{-1}c(g^{-1})^{-1}c(g^{-1}),b^*)
=Q\circ \pi((c(g^{-1})c(g))^{-1}c(g^{-1}),b^*)\\
&=Q\circ \pi(c(g^{-1}),\overline{\gamma(c(g^{-1})c(g))}b^*)=Q\circ \pi(c(g^{-1}),\overline{\gamma(\sigma(g^{-1},g))}b^*),
\end{align*}
which is $\tau_{g^{-1}}(b^{*_{\gamma\circ\sigma}})$ because $\sigma(g^{-1},g)=\sigma(g,g^{-1})$. The representation $\tau_e$ of $B(\gamma\circ\sigma)_e=B_e$ is nondegenerate because all the maps in the composition
\[
B_e\overset{=}{\longrightarrow}(q^*B)_e\overset{\pi_e}{\longrightarrow} C^*(q^*B)\overset{Q}{\longrightarrow}C^*(q^*B)(\gamma) 
\] 
are. Thus $\tau$ is a representation of the Fell bundle $B(\gamma\circ\sigma)$, and hence induces a nondegenerate homomorphism $\tau:C^*(B(\gamma\circ\sigma))\to C^*(q^*B)(\gamma)$.

We claim that $\tau\circ\rho^\gamma=Q$. For $(h,b)\in q^*B$, we compute
\begin{align*}
\tau\circ\rho^\gamma(\pi_h(h,b))
&=\tau(\gamma(hc(h)^{-1})\pi_{q(h)}^\gamma(b))
=\gamma(hc(h)^{-1})\tau_{q(h)}(b)\\
&=\gamma(hc(h)^{-1})Q(\pi_{c(h)}(c(h),b))\\
&=Q(\pi_{hc(h)^{-1}c(h)}(hc(h)^{-1}c(h),b)) \text{ (by \eqref{propQ})}\\
&=Q(\pi_h(h,b)).
\end{align*}
Thus $\tau\circ\rho^\gamma=Q$, as claimed. This implies that $\ker\rho^\gamma\subset \ker Q=I_\gamma$, and hence that the induced homomorphism $\phi_\gamma$ of $C^*(q^*B)(\gamma)$ onto $C^*(B(\gamma\circ\sigma))$ is injective.
\end{proof}

As in the discussion preceding Lemma~\ref{idfibres}, we now have an upper-semicontinuous bundle over $\widehat Z$ with fibres $C^*(B(\gamma\circ\sigma))$. The next lemma will help us prove that the bundle is continuous. Again the idea of the proof goes back at least as far as Anderson and Paschke \cite{AP}, and a similar argument was used by Rieffel in \cite[Theorem~2.5]{Rie}.

\begin{lem}\label{lsc}
Suppose that $G$ is amenable, and consider the homomorphisms $\rho^\gamma$ constructed in Lemma~\ref{idfibres}. Then for each $a\in C^*(q^*B)$, the function $\gamma\mapsto \|\rho^\gamma(a)\|$ is lower semicontinuous from $\widehat Z$ to $[0,\infty)$.
\end{lem}

\begin{proof}
For each $g\in G$, $B_g$ is a right Hilbert module over the $C^*$-algebra $B_e$, with inner product $\langle b,c\rangle:=b^*c$. Hence so is the Hilbert $B_e$-module direct sum $\ell^2(B):=\bigoplus_{g\in G}B_g$. For each $\gamma\in \widehat Z$, $C^*(B(\gamma\circ\sigma))$ has a regular representation $\lambda^\gamma$ in the $C^*$-algebra $\LL(\ell^2(B))$ of adjointable operators, which is given in terms of the original Fell bundle structure on $B$ by
\[
(\lambda^\gamma(b)\xi)(g_2)=\gamma\circ\sigma(g_1,g_1^{-1}g_2)b\xi(g_1^{-1}g_2)\text{ for $b\in B(\gamma\circ\sigma)_{g_1}=B_{g_1}$}
\]
(see \cite[\S2]{E}). Since $G$ is amenable, it follows from \cite[Theorem~4.6]{E} that the Fell bundle $B(\gamma\circ\sigma)$ is amenable, and hence $\lambda^\gamma$ is faithful on $C^*(B(\gamma\circ\sigma))$. Thus it suffices for us to prove that $\gamma\mapsto \|\lambda^\gamma\circ\rho^\gamma(a)\|$ is lower semi-continuous.

For $\xi\in \Gamma_c(B)\subset \ell^2(B)$ and $(h,b)\in q^*B$, we compute
\[
\lambda^\gamma\circ\rho^\gamma(h,b)(\xi)(g)=\gamma\circ\sigma(q(h),q(h)^{-1}g)\gamma(hc(h)^{-1})b\xi(q(h)^{-1}g),
\]
where again the product of $b\in (q^*B)_h=B_{q(h)}$ and $\xi(q(h)^{-1}g)$ is taken in the original bundle $B$. We suppose that $\gamma_i\to \gamma$ in $\widehat Z$, which means that $\gamma_i\to\gamma$ pointwise on $Z$. Then for every $x\in \Gamma_c(q^*B)$ and $\xi\in \Gamma_c(B)$, we have $\lambda^{\gamma_i}\circ\rho^{\gamma_i}(x)\xi\to  \lambda^{\gamma}\circ\rho^{\gamma}(x)\xi$ in norm in $\ell^2(B)$. An $\epsilon/3$ argument shows that $\lambda^{\gamma_i}\circ\rho^{\gamma_i}(a)\to  \lambda^{\gamma}\circ\rho^{\gamma}(a)$ in the strong operator topology. The continuity of the norm topology on $\ell^2(B)$ implies that for each $r$,
\[
\{T\in \LL(\ell^2(B)):\|T\|\leq r\}=\{T:\|T\xi\|\leq r\text{ for all $\xi$ with $\|\xi\|=1$}\}
\]
is closed in the strong operator topology. Thus $\{\gamma:\|\lambda^{\gamma}\circ\rho^{\gamma}(a)\|\leq r\}$ is closed in $\widehat Z$ for every $r\in [0,\infty)$, which says precisely that $\gamma\mapsto\|\lambda^{\gamma}\circ\rho^{\gamma}(a)\|$ is lower semicontinuous.
\end{proof}

We can now state a more concrete version of Theorem~\ref{existbdle1}.

\begin{thm}\label{existbdle2}
Suppose that $p:B\to G$ is a Fell bundle over an amenable group $G$, $\sigma\in Z^2(G,Z)$ is a cocycle with values in an abelian group $Z$, and $q:H\to G$ a corresponding central extension of $G$ by $Z$. For each $\gamma\in \widehat Z$, let $\rho^\gamma$ be the surjection of $C^*(q^*B)$ onto $C^*(B(\gamma\circ\sigma))$ constructed in Lemma~\ref{propQ}, and for $a\in C^*(q^*B)$ define $\widehat a(\gamma):=\rho^\gamma(a)$. Then there is a continuous $C^*$-bundle $F$ over $\widehat Z$ with fibres $F_\gamma=C^*(B(\gamma\circ\sigma))$ such that $a\mapsto \widehat a$ is an isomorphism of $C^*(q^*B)$ onto $\Gamma(E)$.
\end{thm}

\begin{proof}
As in the discussion preceding Lemma~\ref{propQ}, we know that there is an upper semicontinuous $C^*$-bundle $E$ with fibres $E_\gamma=C^*(q^*B)(\gamma)$ and an isomorphism $\phi:C^*(q^*B)\to \Gamma(E)$ such that $\phi(a)=a(\gamma):=a+I_\gamma$. We take $F:=\bigsqcup_{\gamma\in \widehat Z}C^*(B(\gamma\circ\sigma))$, define $p:F\to \widehat Z$ by taking $p\equiv \gamma$ on $F_\gamma=C^*(B(\gamma\circ\sigma))$, and consider the algebra
\[
\{\widehat a:\gamma\mapsto \rho^\gamma(a):a\in C^*(q^*B)\}
\]
of fields of $F$. Lemma~\ref{propQ} says that $\rho^\gamma$ induces an isomorphism of $C^*(q^*B)(\gamma)$ onto $C^*(B(\gamma\circ\sigma))$, and hence $\gamma\mapsto \|\rho^\gamma(a)\|=\|a(\gamma)\|$ is upper semicontinuous. Since $G$ is amenable, Lemma~\ref{lsc} implies that $\gamma\mapsto \|\rho^\gamma(a)\|$ is continuous. Now Theorem~C.25 of \cite{tfb2} implies that there is a unique topology on $F$ such that $p:F\to \widehat Z$ is a continuous $C^*$-bundle and such that the functions $\widehat a$ belong to $\Gamma(F)$. Proposition~C.24 of \cite{tfb2} implies that $a\mapsto \widehat a$ is an isomorphism of $C^*(q^*B)$ onto $\Gamma(F)$. 
\end{proof}

\begin{rmk}\label{KPSex}
When we apply Theorem~\ref{existbdle2} to the Fell bundle over $\Z^k$ associated to the degree functor on a $k$-graph $\Lambda$ and any cocycle $\sigma\in Z^2(\Z^k,Z)$, we recover the continuous bundle of \cite[Corollary~3.3]{KPS3} for the categorical cocycle $c_\sigma$ of \eqref{defcatcocycle}. This covers the examples in \cite[\S3]{KPS3}. 
\end{rmk}

\begin{rmk}\label{repgp}
For every discrete group $G$ there is a particularly interesting cocycle to which we can apply Theorem~\ref{existbdle2}. The cohomology group $H^2(G,\T)$ is a compact abelian group in a topology inherited from the topology of pointwise convergence on cocyles \cite[Corollary~1.3]{PR}. The dual group $H^2(G,\T)^\wedge$ is a discrete group, and there is a canonical central extension 
\[
0\longrightarrow H^2(G,\T)^\wedge\longrightarrow H\longrightarrow G\longrightarrow e
\]
such that the inflation map $\Inf:H^2(G,\T)\to H^2(H,\T)$ is the zero map. Moore called such groups $H$ \emph{splitting extensions} of $G$ \cite[\S II.3]{M}; he proved existence of the extension $H$ in \cite[\S III.2]{M}, and this is also proved as part of \cite[Corollary~1.3]{PR}. Applying Theorem~\ref{existbdle2} to this central extension gives a continuous $C^*$-bundle whose fibres are all the possible deformations $C^*(B(\sigma))$. When $B=G\times\C$ is the trivial Fell bundle, we obtain a continuous bundle over $H^2(G,\T)$ whose fibres are all the twisted group algebras $C^*(G,\sigma)$, as in \cite[\S1]{PR}.

When $G=\Z^k$, the map $A\mapsto [\sigma_A]$ induces an isomorphism of $\T^{k(k-1)/2}$ onto $H^2(\Z^k,\T)$, and we obtain continuous $C^*$-bundles over $\T^{k(k-1)/2}$. When $B=\Z^k\times\C$ is the trivial Fell bundle, we recover the bundle over $\T^{k(k-1)/2}$ whose fibres are the rotation algebras (as in \cite[Example~3.4]{KPS3}).
\end{rmk}

\appendix

\section{Coactions and Fell bundles}\label{app:Q}

Quigg proved in \cite{Q} that Fell bundles over a discrete group $G$ are essentially the same thing as coactions of $G$ on $C^*$-algebras. (We stress that it is crucial that $G$ is discrete: it has long been known that the $C^*$-algebras of Fell bundles over locally compact groups carry natural coactions, but also that not all coactions arise this way \cite[Examples~2.3(6)]{LPRS}.) Since the proofs in \cite{Q} use some fairly hard-core coaction techniques, we include a relatively elementary proof.

If $G$ is a discrete group, we write $u$ for the universal unitary representation of $G$ in its full group $C^*$-algebra $C^*(G)$. All tensor products of $C^*$-algebras that follow are the spatial tensor products studied, for example, in \cite[\S B.1]{tfb}.

Suppose first that $B$ is a Fell bundle over a discrete group $G$, and $\pi$ is the universal representation of $B$ in $C^*(B)$. Then $\delta_g(b):=\pi_g(b)\otimes u_g$ gives a representation $\{\delta_g\}$ of $B$ in $C^*(B)\otimes C^*(G)$, and hence there is a homomorphism $\delta:C^*(B)\to C^*(B)\otimes C^*(G)$ such that $\delta(b)=\pi_g(b)\otimes u_g$ for $b\in B_g$. Since ${\id}\otimes \delta_G(b\otimes u_g)=b\otimes u_g\otimes u_g$, $\delta$ satisfies the coaction identity $(\delta\otimes {\id})\circ\delta=({\id}\otimes \delta_G)\circ\delta$. For $b\in B_g$ we have
\[
b\otimes u_h=(b\otimes u_g)(1\otimes u_{g^{-1}h})=\delta(b)(1\otimes u_{g^{-1}h}),
\]
and hence $\delta$ is nondegenerate in the sense that $\{\delta(a)(1\otimes x):a\in C^*(B),\;x\in C^*(G)\}$ spans a dense subspace of $C^*(B)\otimes C^*(G)$.

Quigg's theorem is the following converse. (He assumed in  \cite[Theorem~3.8]{Q} that his coaction was normal, but that was only needed for the second part of his theorem, which identifies $A$ with the reduced algebra $C_r^*(B)$. There is also a version of this second part for maximal coactions which identifies $A$ with the full algebra $C^*(B)$ \cite[Proposition~4.2]{EKQ}.) 


\begin{thm}[Quigg]\label{Quiggs}
Suppose that $\delta$ is a nondegenerate coaction of a discrete group $G$ on a $C^*$-algebra $A$. For $g\in G$, define
\[
B_g:=\{a\in A:\delta(a)=a\otimes u_g\},
\]
set $B:=\bigsqcup B_g$, and define $p:B\to G$ by $p(b)=g$ for $b\in B_g$. Then with the norms and operations inherited from $A$, $p:B\to G$ is a Fell bundle, and $\bigcup_{g\in G}B_g$ spans a dense $*$-subalgebra of $A$.
\end{thm}

\begin{proof}
Since $\delta$ is a homomorphism of $C^*$-algebras, it is continuous, and each $B_g$ is a closed subspace of $A$. Since $u$ is a homomorphism we have $B_gB_h\subset B_{gh}$ and $B_g^*=B_{g^{-1}}$, and the other properties in \eqref{FBb} and~\eqref{FBc} follow from the corresponding properties of $A$. Property~\eqref{FBd} follows from spectral permanence: because $B_e$ is a $C^*$-subalgebra of $A$,
\[
b^*b\geq 0\text{ in }A\Longrightarrow \sigma_A(b^*b)\subset [0,\infty)
\Longrightarrow \sigma_{B_e}(b^*b)\subset [0,\infty)\Longrightarrow 
b^*b\geq 0\text{ in }B_e.
\]

Before proving the assertion about density, we review some properties of coactions. For $g\in G$, we write $\tau_g$ for the bounded linear functional on $C^*(G)$ given in terms of the usual basis $\{\xi_g\}$ for the left-regular representation by $\tau_g(x)=(\lambda(x)\xi_e\,|\,\xi_g)$, and note that $\tau_g(x)=\tau_e(u_g^*x)$. Each $\tau_g$ induces a bounded linear map ${\id}\otimes{\tau_g}:A\otimes C^*(G)\to A$ such that ${\id}\otimes\tau_g(a\otimes x)=\tau_g(x)a$ (in the coaction literature, this is the slice map $S_{\tau_g}$ discussed in \cite[\S A.4]{EKQR}, for example). 

We next claim that $({\id}\otimes\tau_g)(\delta(a))$ belongs to $B_g$ for every $a\in A$. We have $\tau_g(u_h)=0$ for $h\not =g$ and $\tau_g(u_g)=1$. Thus for $x=\sum_h{c_h}u_h$ we have
\[
({\id}\otimes{\tau_g})(\delta_G(x))=({\id}\otimes {\tau_g})\Big(\sum_{h}c_h(u_h\otimes u_h)\Big)=c_gu_g=\tau_g(x)u_g,
\]
and this extends to $x\in C^*(G)$ by continuity. Thus
\[
({\id}\otimes{\id}\otimes{\tau_g})\circ({\id}\otimes{\delta_G})(a\otimes x) =a\otimes \tau_g(x)u_g=({\id}\otimes\tau_g)(a\otimes x)\otimes u_g,
\]
and this extends to elements of $A\otimes C^*(G)$ by linearity and continuity. Thus for $a\in A$, the coaction identity gives
\begin{align*}
\delta(({\id}\otimes{\tau_g})(\delta(a)))&=({\id}\otimes{\id}\otimes {\tau_g})\circ({\delta}\otimes{\id})(\delta(a))\\
&=({\id}\otimes{\id}\otimes {\tau_g})\circ({\id}\otimes{\delta_G})(\delta(a))\\
&=({\id}\otimes\tau_g)(\delta(a))\otimes u_g,
\end{align*}
which says that $({\id}\otimes\tau_g)(\delta(a))$ belongs to $B_g$, as claimed.

To prove density, we let $a\in A$. By nondegeneracy of $\delta$, we can approximate $a\otimes 1$ in $A\otimes C^*(G)$ by a finite sum $\sum_g (1\otimes u_{g}^*)\delta(a_g)$. Then we have
\[
a={\id}\otimes\tau_e(a\otimes 1)\sim\sum_g{\id}\otimes\tau_e((1\otimes u_{g}^*)\delta(a_g))
=\sum_g{\id}\otimes\tau_g(\delta(a_g)),
\]
which belongs to $\lsp(\bigcup_g B_g)$.
\end{proof}

\section{Fell bundles as gradings}\label{app:grade}

Suppose that $G$ is a group. A ring $R$ is \emph{$G$-graded} if there are additive subgroups $R_g$ such that $R_gR_h\subset R_{gh}$ and ``$R$ is the direct sum of the $R_g$'', which means that each $r$ has a unique expansion as a finite sum $r=\sum_{g\in F}r_g$ with $r_g\in R_g$. We then call the $r_g$ the \emph{homogeneous components} of $r$. 

We suggest, taking the idea of Exel \cite{E} further, that Fell bundles are the right analogue of gradings for $C^*$-algebras.  We think Quigg's theorem provides good support for this view, and allows an alternative description in terms of coactions. The only catch is that the axioms of a Fell bundle do not look like a definition of a grading because there is no mention of decompositions into homogeneous components. For any Fell bundle $p:B\to G$, we can find candidates for the homogeneous components of elements of $C^*(B)$ using the dual coaction $\delta$ of $G$ (using notation from the proof of Theorem~\ref{Quiggs}, we take $b_g:=({\id}\otimes\tau_g)(\delta(b))$\,), but  it is not obvious how to recover an element from its components. Obviously one would like it to be as some kind of infinite sum, but there is no reason to expect this series to converge in the norm of $C^*(B)$ or $C^*_r(B)$. To see the difficulties, consider the group $C^*$-algebra $C^*(\Z)$, which is the $C^*$-algebra of the trivial Fell bundle $\C\times \Z$ over $\Z$. The Fourier transform gives an isomorphism of $C^*(\Z)$ onto $C(\T)$, and then the homogeneous components of a function $f\in C(\T)$ are the functions $z\mapsto \hat f(n)z^n$, where $\hat f(n)$ is the $n$th Fourier coefficient of $f$. Recovering a continuous function $f$ as the sum of its Fourier series is delicate analytically: to get uniform convergence of the Fourier series (that is, convergence in the norm of $C(\T)$\,), we need to know that the function $f$ is smooth. So it seems unlikely that we can achieve recovery for arbitrary groups.

In this appendix, we show how the connection with Fourier series can offer us a way to recover an element from its homogeneous components. Many $C^*$-algebras, including graph algebras, carry natural gauge actions of $\T^k$, and these actions traditionally play the role of the gradings on their algebraic counterparts, such as Leavitt path algebras. For functions on tori, Fourier series have some especially nice properties, and in particular we can accelerate their convergence using summability methods. A classical theorem of Fej{\'e}r, for example, asserts that for every continuous function $f$ on $\T$, the C{\'e}saro means of the Fourier series converge uniformly to $f$. 

Suppose that $\alpha$ is a continuous action of the torus $\T^k$ on a $C^*$-algebra $A$. For $a\in A$ and $n\in \Z^k$, we define the\footnote{There are some identifications going on here. When $\alpha$ is the dual action of $\T^k$ on the $C^*$-algebra of a Fell bundle over $\Z^k$, we are implicitly using the isomorphism
\[
A\otimes C^*(\Z^k)\to A\otimes C((\Z^k)^\wedge)=A\otimes C(\T^k)\to C(\T^k,A)
\]
to pull the coaction $\delta$ of $\Z^k$ over to an action of $\T^k$. Under this isomorphism, the slice map ${\id}\otimes \tau_e$ becomes integration against Haar measure, and $a_n$  is the homogeneous component $({\id}\otimes\tau_n)(\delta(a))$ described a few paragraphs ago.} \emph{$n$th Fourier coefficient} of $a$ by
\[
a_n:=\int_{\T^k} z^{-n}\alpha_z(a)\,dz,
\]
where the integral is with respect to the Haar measure on $\T^k$ and is interpreted as in \cite[Lemma~C.3]{tfb}, for example. Since bounded linear maps pull through this integral, it follows from the translation invariance of Haar measure that $a_n$ belongs to the spectral subspace
\[
A_n:=\{b\in A:\alpha_z(b)=z^n\alpha_z(b)\text{ for $z\in \T^k$}\}.
\]
With all the stucture inherited from $A$, $B:=\bigsqcup_{n\in \N^k}A_n$ is a Fell bundle over $\Z^k$, and, as in the discussion preceding Lemma~\ref{matrixunits}, the inclusion maps induce an isomorphism of $C^*(B)$ onto $A$.

Previous authors have studied the relationship between elements of group algebras and crossed products and their Fourier cofficients (see \cite{BC} for a recent survey). But we have not seen the following $C^*$-algebraic version of Fej{\'e}r's theorem.

\begin{prop}
Suppose that $\alpha$ is a continuous action of the torus $\T^k$ on a $C^*$-algebra $A$. For $a\in A$, we set
\begin{align*}
s_n(a)&:=\sum_{-n\leq j \leq n}a_j\text{ for $n\in \N^k$, and }\\
 \sigma_N(a)&:=\frac{1}{\textstyle{\prod_{j=1}^k(N_j+1)}}\sum_{0\leq n\leq N} s_n(a)\text{ for $N\in \N^k$.}
\end{align*}
Then $\|\sigma_N(a)-a\|\to 0$ as $N\to \infty$ in $\N^k$.
\end{prop}

\begin{proof}
As in the classical proof of Fej{\'e}r's theorem, we can describe the C{\'e}saro means $\sigma_N(a)$ in terms of the Fej{\'e}r kernel $F_N:[-\pi,\pi]^k\to \R$ as
\[
\sigma_N(a)=\int_{[-\pi,\pi]^k} F_N(t) \alpha_{e^{it}}(a)\,dt;
\]
for the rectangular partial sums we are using, $F_N(t)$ is the product $\prod_{j=1}^k F_{n_j}(t_j)$ of the one-dimensional Fej{\'e}r kernels
\[
F_{n_j}(t_j)=\frac{2}{n_j+1}\Big(\frac{\sin ((n_j+1)t_j/2)}{\sin(t_j/2)}\Big)^2
\]
(see \cite[vol.~II, page~303, and vol. I, page 88]{Z}\footnote{We thank Misi Kovacs for pointing us towards the last chapter of \cite{Z}.}). Thus for fixed $\delta>0$, we have $F_N(t)\to 0$ uniformly on $[-\pi,\pi]^k\backslash (-\delta,\delta)^k$ as $N\to \infty$ in $\N^k$, and $\int_{[-\pi,\pi]^k}F_N(t)\,dt=1$ for all $N$. For each $\delta>0$ we have
\begin{align*}
\|\sigma_N(a)-a\|&=\Big\|\int_{[-\pi,\pi]^k}F_N(t)\alpha_{e^{it}}(a)\,dt-\int_{[-\pi,\pi]^k}F_N(t)a\,dt\Big\|\\
&\leq \int_{(-\delta,\delta)^k}F_N(t)\|\alpha_{e^{it}}(a)-a\|\,dt +2\|a\|\int_{[-\pi,\pi]^k\backslash (-\delta,\delta)^k}F_N(t)\,dt;
\end{align*}
we can then make the first integral small by making $\delta$ small using the continuity of $t\mapsto \alpha_{e^{it}}(a)$, and the second small for this fixed $\delta$ by taking $N$ large. Thus $\|\sigma_N(a)-a\|\to 0$, as required.
\end{proof}

\section{Gradings of Kumjian--Pask algebras}\label{app:KP}

We can use the ideas of the previous section to give a direct argument which derives the $\Z^k$-grading on a Kumjian--Pask algebra $\KP_K(\Lambda)$ from the universal property. The existence of the grading has previously been deduced (at the foot of page~3620 in \cite{APCaHR}) from the construction of $\KP_K(\Lambda)$ as a quotient of a free algebra, as was previously done for Leavitt path algebras in \cite[Lemma~1.7]{AAP}. Here we deduce existence by applying the following proposition to the degree functor. Proposition~\ref{gradeKP} would also apply to the functors used in the theory of graph coverings in \cite{PQR}.

\begin{prop}\label{gradeKP}
Suppose $\Lambda$ is a row-finite $k$-graph with no sources, that $K$ is a field, and $\eta:\Lambda\to G$ is a functor with values in a group $G$. Then there is a $G$-grading of the Kumjian--Pask algebra with
\[
\KP_K(\Lambda)_g:=\lsp\{s_\lambda s_{\mu^*}:\eta(\lambda)\eta(\mu)^{-1}=g\}.
\]
\end{prop}

\begin{proof}
Suppose $s_\lambda s_{\mu^*}\in \KP_K(\Lambda)_g$ and $s_\sigma s_{\tau^*}\in \KP_K(\Lambda)_h$. Then as in the proof of Proposition~\ref{labeltoFell},
\[
(s_\lambda s_{\mu^*})(s_\sigma s_{\tau^*})=\sum_{(\alpha,\beta)\in \Lambda^{\min}(\mu,\sigma)}s_{\lambda\alpha}s_{(\tau\beta)^*};
\]
belongs to $B_{gh}$ because $\eta$ is a functor. It remains for us to show that the subspaces $\{B_g:g\in G\}$ give a direct-sum decomposition of the additive group of $\KP_K(\Lambda)$. They certainly span $\KP_K(\Lambda)$, so we need to see that they are independent. To do this, we borrow an argument from the theory of coactions.

Let $KG:=\{\sum_{g\in G} c_g g: c_g\in K\}$ be the group algebra of $G$. Then the elements $\{g:g\in G\}$ are a basis for $KG$, and hence there is a linear transformation $\tau:KG\to K$ such that $\tau(g)=0$ if $g\not=e$ and $\tau(e)=1$. For each $g\in G$, we define $\tau_g(y)=\tau(g^{-1}y)$. The elements $s_\lambda\otimes \eta(\lambda)$ form a Kumjian--Pask family in $\KP_K(\Lambda)\otimes_K KG$, and hence there is a homomorphism $\pi_{s\otimes\eta}:\KP_K(\Lambda)\to \KP_K(\Lambda)\otimes_K KG$ such that $\pi_{s\otimes\eta}(s_\lambda)=s_\lambda\otimes \eta(\lambda)$.  Since $(x\otimes c)\mapsto x\tau_g(c)$ is bilinear from $\KP_K(\Lambda)\times KG$ to $\KP_K(\Lambda)$, there is a well-defined linear map ${\id}\otimes\tau_g$ from $\KP_K(\Lambda)\otimes_K KG$ to $\KP_K(\Lambda)$ such that ${\id}\otimes\tau_g(x\otimes h)=x\tau_g(h)$. Now a straightforward calculation shows that
\[
({\id}\otimes \tau_g)\circ\pi_{ s\otimes \eta}(s_{\lambda}s_{\mu^*})
=\begin{cases}
s_{\lambda}s_{\mu^*}&\text{if $\eta(\lambda)\eta(\mu)^{-1}=g$}\\
0&\text{otherwise.}
\end{cases}
\]
Now if $\sum_{h\in F}b_h$ is a finite sum  with $b_h\in B_h$ and $\sum_{h\in F}b_h=0$ in $\KP_K(\Lambda)$, we have 
\[
b_g=({\id}\otimes \tau_g)\circ\pi_{ s\otimes \eta}\Big(\sum_{h\in F}b_h\Big)=0 \text{ for all $g\in F$.}
\]
Thus the spaces $B_g$ are independent.
\end{proof}

\end{document}